\documentclass[reqno,11pt]{amsart}

\usepackage{stmaryrd}
\usepackage[top=2.0cm,bottom=2.0cm,left=2cm,right=2cm]{geometry}
\usepackage{amsthm,amsmath,amssymb,dsfont}
\usepackage{mathrsfs,amsfonts,functan,extarrows,mathtools}
\usepackage[colorlinks]{hyperref}

\usepackage{marginnote}
\usepackage{xcolor}
\usepackage{stmaryrd}
\usepackage{esint}
\usepackage{graphicx}
\usepackage{bm}
\usepackage{soul}

\newtheorem{theorem}{Theorem}[section]
\newtheorem{proposition}{Proposition}[section]

\newtheorem{corollary}{Corollary}[section]
\newtheorem{remark}{Remark}[section]
\newtheorem{lemma}{Lemma}[section]

\numberwithin{equation}{section}
\allowdisplaybreaks

\def\R{\mathbb{R}}
\def\g{\gamma}

\def\P{\mathbf{P}}
\def\I{\mathbf{I}}
\def\v{\varepsilon}
\def\div{\mathrm{div}}

\def\l{\langle}
\def\r{\rangle}

\def\up{\textup}
\def\n{\nabla}
\def\p{\partial}

\def\L{\mathcal{L}}

\def\le{\lesssim}
\def\a{\alpha}
\def\b{\beta}

\def\le{\lesssim}
\def\D{\mathcal{D}}
\def\c{\cdot}
\def\g{\gamma}

\def\E{\mathcal{E}}
\def\M{\sqrt{M}}

\newcounter{wronumber}

\begin{document}
\title[The small Deborah number limit for the compressible fluid-particle flows]
{The small Deborah number limit for the compressible fluid-particle flows}

\author[Zhendong Fang]{Zhendong Fang}
\address[Zhendong Fang]
        {\newline School of Mathematics and Information Science, Guangzhou University, Guangzhou 510006, P. R. China}
\email{zdfang@gzhu.edu.cn}

\author[Kunlun Qi]{Kunlun Qi}
\address[Kunlun Qi]
        {\newline Department of Computational Mathematics, Science and Engineering and Department of Mathematics, Michigan State University, East Lansing, MI 48824, USA}
\email{qikunlun@msu.edu; kunlunqi.math@gmail.com}

\author[Huanyao Wen]{Huanyao Wen}
\address[Huanyao Wen]
        {\newline School of Mathematics, South China University of Technology, Guangzhou, 510641, P. R. China}
\email{mahywen@scut.edu.cn}



\begin{abstract}
In this paper, we consider the hydrodynamic limit for the fluid-particle flows governed by the Vlasov-Fokker-Planck equation coupled with the compressible Navier-Stokes equation as the Deborah number tends to zero. 
The proof is based on a formal derivation via the Hilbert expansion around the limiting system, the rigorous justification of which is completed by the refined energy estimates involving the macro-micro decomposition. 
Compared with the existing results obtained by the relative entropy argument ([A.~Mellet and A.~F. Vasseur, Comm. Math. Phys., 281 (2008), pp.~573--596]), the present work extends to a pointwise convergence of the hydrodynamic limits with an explicit rate for the fluid-particle coupled model.
\end{abstract}

\maketitle

\section{Introduction}

\subsection{The model}
\label{subsec:model}

Many natural phenomena are described by the fluid-particle two-phase flow models, including the droplets of spray, diesel engines, sedimentation analysis, biotechnology, medicine, and mineral processes \cite{CBLBPEJSM05, SBRBKT03, WFA85}, where one phase is considered as a suspension of particles within the other phase thought as a fluid. One of the models that has been widely considered is the Vlasov-Fokker-Planck equation coupled with the compressible Navier-Stokes equations (VFP-CNS):
\begin{equation}\label{VFP-CNS-O}
\left\{
\begin{aligned}
&\p_t f+v\c\n_x f = \frac{9\pi}{2R^2\rho_P}\div_v\big[\frac{k\mathcal{T}}{\mathcal{M}}\n_v f+(v-u)f\big],\\[3pt]
&\rho_F\big[\p_t(\rho u) + \div_x(\rho u\otimes u) + \n_x \tilde{P}(\rho)\big] + \tilde{L}u = 6\pi \tilde{\mu} R\int_{\R^3}(v-u)f \,dv,\\[3pt]
&\p_t \rho+\div_x(\rho u) = 0,
\end{aligned}
\right.
\end{equation}
where $f(t, x, v)$ is the density distribution function of the particles, $\rho(t,x)$, $u(t,x)$ denote the density and velocity of the fluid at time $t\geq 0$, position $x=(x_1,x_2,x_3)\in\R^3$, velocity $v=(v_1,v_2,v_3)\in\R^3$ respectively. $\tilde{P}(\rho)=\tilde{A} \rho^\g$ denotes the pressure function with $\tilde{A}>0,\,\g>1$, $\tilde{L}u= - \tilde{\mu} \Delta_x u-(\tilde{\mu} + \tilde{\lambda}) \n_x \up{div}_x u$ is the so-called the Lam$\acute{\text{e}}$ operator with constants $\tilde{\mu} + \tilde{\lambda} > 0$ and $\tilde{\mu} > 0$ being the dynamic viscosity of the fluid. $R$, $\mathcal{M}$ and $\rho_P$ are the radius, mass and associated mass density of one single spherical particle, respectively. $\rho_F$ is the mass density of the fluid, $k$ is the Boltzmann constant, and $\mathcal{T}>0$ is the temperature of the suspension that is assumed to be constant throughout this paper as in \cite{GTJPEVA04-1}.

After re-normalization as in \cite{CG06}, we can deduce the scaled VFP-CNS system in the dimensionless form by introducing $\chi=\frac{\mathscr{P}}{\rho_F U^2}$, the Mach number $\up{Ma}$, Reynolds number $\up{Re}$, and Deborah number $\up{De}$:
\begin{equation}\label{VFP-CNS-0}
\left\{
\begin{aligned}
&\p_t f+\frac{1}{\up{Ma}}v\c\n_x f=\frac{1}{\up{De}}\div_v\big[\n_v f+(v-\up{Ma}\,u)f\big],\\[3pt]
&\p_t(\rho u)+\div_x(\rho u\otimes u)+\chi\n_x P(\rho)+\frac{1}{\up{Re}}L u = \frac{\rho_P}{\rho_F \up{Ma}\,\up{De}}\int_{\R^3}(v-\up{Ma}\,u)f \,dv,\\[3pt]
&\p_t \rho+\div_x(\rho u)=0,
\end{aligned}
\right.
\end{equation}
where $\mathscr{P}$ is a pressure unit, $U$ is a macroscopic velocity unit, $P(\rho)=A\rho^\g$ with $A>0$ and $Lu=\frac{2}{9\tilde{\mu}}\tilde{L}u:=-\mu \Delta_x u-(\mu +\lambda)\n_x\up{div}_x u$ with dimensionless constant $\mu>0,\mu + \lambda>0$. We refer to \cite{RRH75, SY07} for more physical background about the dimensionless analysis as well as the Mach number $\up{Ma}$, Reynolds number $\up{Re}$, and Deborah number $\up{De}$.

In this paper, the small Deborah number limit in the light particles regime is considered for the scaled VFP-CNS system \eqref{VFP-CNS-0} in the sense that
\begin{equation}\label{Scale}
\begin{aligned}
\chi=1,\quad \up{Ma}=\v, \quad \up{Re}=1,\quad \up{De}=\v^2, \quad \frac{\rho_P}{\rho_F}=\v^2.
\end{aligned}
\end{equation}
Then the scaled VFP-CNS system \eqref{VFP-CNS-0}  can be rewritten as
\begin{equation}\label{VFP-CNS}
\left\{
\begin{aligned}
&\p_t f^\v+\frac{1}{\v}v\c\n_x f^\v = \frac{1}{\v^2} \div_v \big[ \n_v f^\v+(v-\v u^\v)f^\v \big],\\[3pt]
&\p_t(\rho^\v u^\v)+\div_x(\rho^\v u^\v\otimes u^\v)+\n_x P(\rho^\v)+L u^\v = \frac{1}{\v}\int_{\R^3}(v-\v u^\v)f^\v \,dv,\\[3pt]
&\p_t \rho^\v+\div_x(\rho^\v u^\v) = 0,\\[3pt]
&(f^\v, u^\v, \rho^\v)|_{t=0} = \big(f^{\v,in}(x,v),u^{\v,in}(x),\rho^{\v,in}(x)\big) \to (M,0,1),\quad \up{as}\, \quad  |x| \to +\infty,
\end{aligned}
\right.
\end{equation}
where $M:=M(v)$ is a global normalized Maxwellian equilibrium given by:
\begin{equation}\label{Maxwellian}
M(v)=\frac{1}{(2\pi)^{\frac{3}{2}}}\up{e}^{-\frac{|v|^2}{2}}.
\end{equation}

Inspired by \cite{GTJPEVA04-1}, taking $\v\to 0$, the Navier-Stokes-Smoluchowski (NSS) equation is formally deduced (see Section \ref{sec:formal_analysis}) for $\big(n_0(t,x), u_0(t,x),\rho_0(t,x)\big)$:
\begin{equation}\label{NSS}
\left\{
\begin{aligned}
&\p_t n_0+\up{div}_x(u_0 n_0)=\Delta_x n_0,\\[3pt]
&\p_t(\rho_0 u_0)+\up{div}_x(\rho_0 u_0\otimes u_0)+\n_x P(\rho_0)+Lu_0+\n_x n_0=0,\\[3pt]
&\p_t \rho_0+\div_x(\rho_0 u_0)=0,\\[3pt]
&(n_0,u_0,\rho_0)|_{t=0}=\big( n_0^{in}(x),u_0^{in}(x),\rho_0^{in}(x) \big) \to (1,0,1),\quad \up{as}\, \quad |x|\to +\infty.
\end{aligned}
\right.
\end{equation}

\subsection{Previous results and our contributions}
Due to its physical significance, the fluid-particle two-phase model has attracted people's attention for a long time, which can be traced back to the work of Caflish-Papanicolaou in \cite{CP1983}.
In the seminal work \cite{GTJPEVA04-1, GTJPEVA04}, Goudon-Jabin-Vasseur first established a systematic framework to study the particle-fluid two-phase flow, i.e., the VFP equation coupled with incompressible Navier-Stokes equation, where their study involved two distinct scalings for the light particles (density of the particle is much less than the fluid density, i.e., $\rho_P/\rho_F = O(\v^2)$) and fine particles (density of the particle is of the same order as gas density, i.e., $\rho_P \approx \rho_F$), as informed by dimensional analysis; our model \eqref{VFP-CNS} follows the scaling proposed in \cite{GTJPEVA04-1}, though the compressible Navier-Stokes equation is considered to be coupled with the VFP equation.

To clearly state our contributions in this paper, we begin with a comprehensive overview of the fluid-particle model, encompassing the previous results of both its well-posedness and hydrodynamic limit; furthermore, the mathematical challenges and contributions of our work are illustrated for comparison.\\[-8pt]

\textit{Previous results for ``well-posedness" of the fluid-particle model:}
There are many results concerning the well-posedness of the VFP equation coupled with the compressible Navier-Stokes system. Mellet-Vasseur studied the existence of global weak solutions to the VFP-CNS system with Dirichlet or reflection boundary conditions in \cite{MAVA07}.
Li-Mu-Wang in \cite{LFCMYMWDH17} obtained the global well-posedness of a strong solution when the initial data is a small perturbation of some given equilibrium, along with the algebraic convergence rate of a solution toward the equilibrium, and similar results can be extended to the VFP equation coupled with the non-isentropic CNS equation by Mu-Wang in \cite{MYMWDH20}.
Recently, the existence of a global-in-time strong solution to the VFP-CNS system with specular reflection boundary conditions was proved by Li-Liu-Yang in \cite{LHLLSQYT22}.
Besides, in the presence of various boundary conditions, the global existence of weak solutions to the incompressible case has been shown in \cite{CJACYPK16, LHLLSQYT22, WDHYC15, YC13}, and we refer the readers to our previous work \cite{FQW2023} and the references therein for more results regarding the VFP-INS model. As for VFP equation coupled with the Euler equations, Carrillo-Duan-Moussa \cite{CJADRJMA11} and Duan-Liu \cite{DL13} proved the existence and large time behaviors of the classical solutions to the Cauchy problem, while the stability and asymptotic analysis of such a model was studied by Carrillo-Goudon in \cite{CG06}. \\[-10pt]

\textit{Previous results for ``hydrodynamic limit" and ``limiting two-phase model":}
Another important research field of the fluid-particle model is its hydrodynamic limit. Carrillo-Goudon in \cite{CG06} first studied the VFP equation coupled with the compressible Euler equations, where they formally derived the hydrodynamic model in the so-called ``bubbling" and ``flowing" regime with different scalings.
The rigorous justifications for hydrodynamic limit in these two regimes were completed by Mellet-Vasseur \cite{MV08} and Ballew \cite{JB20}, respectively. It is worth mentioning that both results relied on the relative entropy method (also called the ``modulated energy" method for different asymptotic problems), which is reminiscent of the weak-strong uniqueness principle by Dafermos \cite{Dafermos1996} and Lions \cite{Lions1998}. Recently, Choi-Jung also applied a similar strategy to study the case of VFP equation coupled with compressible Navier-Stokes equations with a density-dependent viscosity \cite{YPCJJ20}.
The hydrodynamic limit results of the VFP equation coupled with the incompressible Navier-Stokes equations were established by Goudon-Jabin-Vasseur in \cite{GTJPEVA04-1, GTJPEVA04} via the weak compactness and relative entropy argument. In our previous work \cite{FQW2023}, we further rigorously justified the hydrodynamic limit with an explicit convergent rate by designing a new expansion form and applying the refined energy estimate.

Furthermore, to rigorously justify the hydrodynamic limit in a stronger topology (e.g., pointwise), it is essential to rely on the well-posedness and regularity theory of the limiting system. In fact, the limiting two-phase system in the so-called “bubbling” regime \cite[Section 4.3]{CG06} corresponds to the NSS system \eqref{NSS}. Within this framework, Huang-Ding-Wen established the local existence and uniqueness of the strong solutions to the compressible NSS system in \cite{HDW16}, which was later extended to global-in-time solutions for initial data with small energy near equilibrium in \cite{DHW17}. In the presence of a ``small” external potential, the global existence of classical solutions was further proved by Ding-Huang-Li in \cite{DHL19}. For additional results on the NSS system, we refer the reader to \cite{B14, CKT11, FZZ12}.
Although this paper primarily focuses on the “bubbling” regime, another class of two-phase limiting models, arising from the “flowing” regime \cite[Section 4.2]{CG06}, has attracted significant attention in recent years. In \cite{VWY2019}, Vasseur-Wen-Yu established the global existence of weak solutions for a bi-fluid two-phase model with a pressure law. This result was further generalized by Novotny-Pokorny in \cite{NP2020} and Wen in \cite{Wen2021}. In \cite{BMZ2019}, Bresch-Mucha-Zatorska proved the global existence of weak solutions to a related two-fluid compressible Stokes system for more general adiabatic index. For further developments on the limiting two-phase system in the “flowing” regime, we refer the reader to \cite{EK2008, EWZ2017, KMN2024}. These results mentioned above motivate our future goal of rigorously justifying the hydrodynamic limit in the “flowing” regime.\\[-10pt]

\textit{Mathematical challenges and our contributions:}
As we discussed above, in the previous results concerning the hydrodynamic limit of the VFP-CNS system \cite{JB20, MV08}, the proof essentially relied on the relative entropy argument.
Albeit successful, the pointwise convergence seems to be tough to obtain through the ``modulated energy" method.
As a consequence, our main purpose in this paper is to rigorously justify the hydrodynamic limit from the VFP-CNS system \eqref{VFP-CNS} to the NSS system \eqref{NSS} in a pointwise sense with an explicit convergence rate.
Specifically speaking, we start with looking for a special class of solutions to the scaled kinetic-fluid coupled system \eqref{VFP-CNS} in the form of Hilbert expansion, where, in contrast with the previous scalings corresponding to the flow regime \cite{MV08} and bubble regime \cite{JB20}, we take the scaling in \eqref{VFP-CNS} from \cite{GTJPEVA04-1} to model the two-phase flow of light particles. The key point here is that the Hilbert expansion is taken around the classical solution to our limiting macroscopic system \eqref{NSS}, hence, the essential part of the proof lies in the establishment of the uniform energy estimate for the remainder system \eqref{Remainder equations}, which is manageable thanks to its less singular and nonlinear property than the original coupled system \eqref{VFP-CNS}.
Note that this sort of strategy and expansion has been applied in justifying the stability of the Boltzmann equation near Couette flows in \cite{DLY2025}, the hydrodynamic limit of the self-organized kinetic equation coupled with a fluid equation in \cite{JNXLJZTF16}, as well as the VFP equation coupled with the incompressible Navier-Stokes equations in our preceding work \cite{FQW2023}. However, even though the decomposition framework has proven highly effective for incompressible flows, extending this methodology to the compressible regime is highly non-trivial. The critical challenges, along with the novelties of our strategies, are illustrated below:
\begin{itemize}

\item In contrast with the incompressible case, where the density is a constant, the appearances of the extra continuity equation and the density-dependent pressure function in the coupled compressible Navier-Stokes equation bring more difficulties in our analysis.
The notable challenge arises from the strong couplings between the momentum equations and the continuity equation, leading to severe singularity in the remainder system such that the usual expansion form of $(f^\v, u^\v, \rho^\v)$, directly derived from the classical Hilbert expansion, would no longer be sufficient to directly handle the singular terms via dissipation. 
Hence, we construct the expansion \eqref{Special form solution} based on our observation and formal analysis (see Section \ref{sec:formal_analysis}).
Note that the primary distinction of the expansion in the compressible case lies in the extra density expansion of $\rho^\v = \rho_0 + \v\rho = (1+h_0) + \v\rho$. 
Still, balancing the nonlinear pressure terms gives rise to a distinctive singular term $\frac{1}{\v} R_2$ in \eqref{L-R0-R3} in the compressible case,

To address this, we introduce a new auxiliary function (e.g., $B^{\g}_{\a}$ in Lemma \ref{estimate of Presure}) 
\[
B^\a_{\g} = \p_x^{\a} \left[(1+h_0+\v\rho)^{\g}\right] -\p_x^\a\left[(1+h_0)^{\g}\right] - \g\v(1+h_0+\v\rho)^{{\g}-1}\p_x^\a\rho,
\]
and explicitly calculate its spatial derivatives before applying Taylor's expansion to each resulting component, where an $\mathcal{O}(\v)$ term naturally emerges and exactly cancels out the singularity of $\frac{1}{\v}$ inherent in the pressure term. This crucial cancellation enables us to uniformly control the nonlinearity from the pressure term and successfully close the energy estimates.
This newly-designed expansion profile provides the ``optimal" order of convergence in the sense of aligning with the order of singularity in the remainder system \eqref{Remainder equations}. 
It is worth mentioning that the applications of such a novel strategy, i.e., re-designing the expansion form to match the singular behavior and close the energy estimate, are expected to be extensively fruitful, especially when the high-order singularity exists. \\[-10pt]

\item In addition, our study extends the previous convergence results by the relative entropy method \cite{JB20, MV08} to a pointwise convergence (see Corollary \ref{Conergence}), thanks to the high-regularity estimates.
To this end, we apply dedicated energy estimates to the reminder system \eqref{Remainder equations} by proposing the refined energy/dissipation structures that incorporate higher regularity in both spatial and velocity variables (see Section \ref{subsec:energy}). More specifically, our refined energy/dissipation functionals \eqref{energy functionals and energy dissipation functional} are designed based on the well-established macro-micro decomposition \cite{DFT10}
such that they consist of contributions from both macroscopic and microscopic parts, which also perfectly match our new expansion form (see \eqref{Special form solution}). The refined total energy estimate of the reminder system \eqref{Remainder equations} can be found in Proposition \ref{A priori estimates}.

\end{itemize}

\subsection{Notations}

The notations that will be used throughout this paper are introduced as follows:

\textup{(i)} $A_1\le B_1$ stands for $ A_1\leq C B_1$ with generic constant $C>0$. $A_1\sim B_1$ stands for $C_1A_1\leq B_1\leq C_2A_1$ with some generic constants $C_1,C_2 > 0$.

\textup{(ii)} For multi-indices $\alpha = (\alpha_1, \alpha_2, \alpha_3)$ and $\beta = (\beta_1, \beta_2, \beta_3)$, we denote
\begin{equation*}
    \partial_x^\alpha = \partial_{x_1}^{\alpha_1} \partial_{x_2}^{\alpha_2} \partial_{x_3}^{\alpha_3}, \quad \partial_v^\beta = \partial_{v_1}^{\beta_1} \partial_{v_2}^{\beta_2} \partial_{v_3}^{\beta_3}.
\end{equation*}

\textup{(iii)}
For $d,e\in\mathbb{N}$, we denote the following inner-product and the spaces:
\begin{equation*}
\begin{aligned}
\l u,w\r_x =& \int_{\R^3} uw \,dx,\,  \quad  \l f,g\r_v=\int_{\R^3} fg \,dv,\, \quad \|u\|_{L^2_x}= \l u,u\r_x^{\frac{1}{2}},\, \quad \|f\|_{L^2_v}= \l f,f\r_v^{\frac{1}{2}}, \\[4pt]
\l f,g \r_{x,v} =& \int_{\R^3} \int_{\R^3} fg \,dvdx,\, \quad \|f\|_{L^2_{x,v}}=\l f,f\r_{x,v}^{\frac{1}{2}},
\end{aligned}
\end{equation*}\\[-10pt]

\begin{equation*}
\begin{aligned}
H^d_{x}:=&\big\{u(x) \ \big| \ \|\p_x^\alpha u\|_{L^2_x} <\infty,\,\up{for any}\,|\a|\leq d\big\},\\[6pt]
H^d_{x,v}:=&\big\{f(x,v) \ \big| \ \|\p_x^\alpha\p_v^\beta f\|_{L^2_{x,v}} <\infty,\,\up{for any}\,|\a|+|\b|\leq d\big\},\\[6pt]
\mathcal{H}^d_{x,v}:=&\big\{f(x,v) \ \big| \ \|\p_x^\alpha\p_v^{\beta+\beta_1} f\|_{L^2_{x,v}} <\infty,\,\up{for any}\,|\a|+|\b|+|\beta_1|\leq d \, \text{and} \, |\beta_1|=1\big\},\\[6pt]
H^d_x H^e_v:=&\big\{f(x,v) \ \big| \ \|\p_x^\alpha\p_v^\beta f\|_{L^2_{x,v}} <\infty,\,\up{for any}\,|\a|\leq d,\,|\b|\leq e\big\}.
\end{aligned}
\end{equation*}\\[-10pt]

\qquad Let $\nu(v)=1+|v|^2$ and denote  $\|\c\|_{\nu}$ by
\begin{equation*}
\|f\|_{\nu} = \Big(\int_{\R^3}\int_{\R^3}|\p_v f(x,v)|^2+|f(x,v)|^2\nu(v) \,dv \,dx\Big)^{\frac{1}{2}}.
\end{equation*}

\textup{(iv)}
By following \cite{DFT10}, the velocity orthogonal projection $\P\up{:}\,L^2_v\to\up{Span}\{\M,v_1\M,v_2\M,v_3\M\}$ is denoted as
\begin{equation*}
\P=\P_0\oplus \P_1,\quad \P_0 g=a\M, \quad \P_1 g=v \c b\M
\end{equation*}
with
$$a=\int_{\R^3} g\M \,dv \quad \text{and} \quad b=\int_{\R^3} g v \M \,dv.$$
\qquad The linearized Fokker-Planck operator $\L$ is denoted as
\begin{equation*}
\L g=-\frac{1}{\M}\up{div}_v\big[M\n_v(\frac{g}{\M})\big],
\end{equation*}
with $\up{Ker} \L=\up{Span}\{\M\}$ in $L^2_v$ norm, and $\L $ can be decomposed by
\begin{equation}\label{Decop-L}
\L g=\L(\I-\P)g+\P_1g.
\end{equation}
\qquad Notice that $\I-\P_0$, $\I-\P$ is self-adjoint in $H^d_{x,v}$, i.e., for any $f,g\in H^d_{x,v}$,
\begin{equation}\label{self-adjoin of I-P}
\begin{aligned}
\l\p_x^\alpha(\I-\P_0)f, \,\p_x^\alpha g\r_{x,v} = \l \p_x^\alpha f, \, \p_x^\alpha(\I-\P_0)g\r_{x,v} \,, \quad \l\p_x^\alpha(\I-\P)f, \,\p_x^\alpha g\r_{x,v} = \l\p_x^\alpha f, \, \p_x^\alpha(\I-\P)g\r_{x,v}\,,
\end{aligned}
\end{equation}
and it can be verified that
\begin{equation}\label{I-P}
\begin{aligned}
(\I-\P_0)(\I-\P)=\I-\P, \quad (\I-\P_0)(\I-\P_0)=\I-\P_0,\quad (\I-\P)(\I-\P)=\I-\P.
\end{aligned}
\end{equation}
\qquad According to \cite{CJADRJMA11}, the dissipative property holds for $\L$, i.e., there exists a constant $c_0>0$ such that
\begin{equation}\label{Poin inequality}
c_0\|(\I-\P)f\|_{\nu}^2+\|b\|^2_{L^2_x}\leq \l \L f, f\r_{x,v}.
\end{equation}

\subsection{Main results}
\label{Main theorem}

Let $h_0(t,x):=\rho_0(t,x) - 1, m_0(t,x):= n_0(t,x) - 1$, the NSS system \eqref{NSS} becomes
\begin{equation}\label{NSS-2}
\left\{
\begin{aligned}
&\p_t m_0+\up{div}_x\big[(1+m_0)u_0\big]=\Delta_x m_0, \\[4pt]
&\p_t[(1+h_0)u_0]+\up{div}_x[(1+h_0) u_0\otimes u_0]+\n_x P(1+h_0)+Lu_0+\n_x m_0=0,\\[4pt]
&\p_t h_0+\up{div}_x[(1+h_0)u_0]=0,\\[4pt]
&(m_0,u_0,h_0)|_{t=0} = (m_0^{in},u^{in}_0,h_0^{in}) \to (0,0,0), \quad \up{as} \quad |x|\to +\infty.
\end{aligned}
\right.
\end{equation}

For the NSS system \eqref{NSS-2} above, we have the following well-posedness result:

\begin{proposition}\label{Solu-NSS}
Assume that the initial data $(m_0^{in},u_0^{in},h_0^{in})$ satisfy\\[5pt]
\up{(i)} $1+\inf_{x\in\R^3} h_0^{in}(x)>0$,\\[5pt]
\up{(ii)} $(m_0^{in},u_0^{in},h_0^{in})\in H^6_x \times H^6_x \times H^6_x$.\\[5pt]
Then, there exists a small constant $\delta_0 > 0$ such that if $\|(m_0^{in},u_0^{in},h_0^{in})\|_{H^6_x}^2 \leq \delta_0$ with $\int u_0^{in}(x)dx=0$, the system \eqref{NSS-2} admits a unique global-in-time solution $(m_0,u_0,h_0)$ satisfying
\begin{equation*}
(m_0,u_0) \in C^0 \left( [0,+\infty); H^6_x \right) \cap C^1 \left( [0,+\infty); H^4_x \right), \quad
h_0 \in C^0 \left( [0,+\infty); H^6_x \right) \cap C^1 \left( [0,+\infty); H^5_x \right).
\end{equation*}
Furthermore, for all $t>0$, there exists a constant $C_0 > 0$ such that
\begin{equation}\label{Macro-part}
\frac{1}{2}\frac{d}{dt} \E_{ma}(t) + C_{0}\D_{ma}(t) \leq 0,
\end{equation}
and
\begin{equation}\label{m0u0h0}
    \|\big( m_0(t,\cdot), u_0(t,\cdot), h_0(t,\cdot) \big) \|_{H^6_x}^2 \le \|(m_0^{in},u_0^{in},h_0^{in})\|_{H^6_x}^2 \le  \delta_0,
\end{equation}
where the macroscopic energy functional $\E_{ma}(t)$ and dissipation functional $\D_{ma}(t)$ are defined in \eqref{energy functionals} and \eqref{energy dissipation functionals}.
\end{proposition}

\begin{remark}
The Proposition \ref{Solu-NSS} can be proved by following similar argument in \cite{DHL19}, where, however, the well-posedness and energy estimates are obtained in the lower-regularity space, i.e.,
\begin{equation*}
(m_0,u_0) \in C^0 \left([0,+\infty); H^3_x \right) \cap C^1 \left([0,+\infty); H^1_x \right), \quad
h_0 \in C^0 \left([0,+\infty);H^3_x \right) \cap C^1 \left([0,+\infty); H^2_x\right)
\end{equation*}
with the corresponding initial conditions $(m_0^{in},u_0^{in},h_0^{in})\in H^3_x \times H^3_x \times H^3_x$, $1+\inf_{x\in\R^3} h_0^{in}(x)>0$, and the ``smallness" of $\|(m_0^{in},u_0^{in},h_0^{in})\|_{H^3_x}$. We are considering to remove the ``smallness" assumption (at least in the two dimensions) in our future work.
\end{remark}

We seek a solution to the VFP-CNS system \eqref{VFP-CNS} in the following form by employing $(m_0,u_0,h_0)$ of the NSS system \eqref{NSS-2}:
\begin{equation}\label{Special form solution}
\begin{aligned}
f^\v= & g_0+\v g_1\M+\v g\M\\
 = &(1+m_0)M + \v[(v\c u_0)(1+m_0)-v\c\n_x m_0]M + \v g\M,\\
u^\v= & u_0 + \v u,\\
\rho^\v = & \rho_0 + \v\rho\\
        = & (1+h_0) + \v\rho,
\end{aligned}
\end{equation}
where $g$, $u$, $\rho$ also depend on $\varepsilon$, but, for simplicity, we will omit the dependence in the notation throughout the paper.

Then, the main theorem is presented as follows:

\begin{theorem}\label{Main-Limits}
Let $(m_0(t,x),u_0(t,x),h_0(t,x))$ be the solution to the NSS equations \eqref{NSS-2} given by Proposition \ref{Solu-NSS}, and assume the initial data of the VFP-CNS system \eqref{VFP-CNS} to be
\begin{equation}\label{Theorem-intial data of VFP-NS}
\begin{aligned}
f^{\v,in}(x,v)=&\ (1+m_0^{in})M+\v\big[(v\c u_0^{in})(1+m^{in}_0)-v\c\n_x m^{in}_0\big]M + \v g^{in}\M,\\
u^{\v,in}(x) = &\ u_0^{in}+\v u^{in},\\
\rho^{\v,in}(x) = &\ 1+h_0^{in}+\v\rho^{in}.
\end{aligned}
\end{equation}
Then, there exist small constants $\v_0,\,\delta > 0$ such that, for any given $\v\in(0,\v_0]$,
the VFP-CNS system \eqref{VFP-CNS} admits a unique global-in-time solution $(f^\v(t,x,v),u^\v(t,x),\rho^\v(t,x))$ in the following form:
\begin{equation}\label{Form of the solution}
\begin{aligned}
f^\v(t,x,v)=&\ (1+m_0)M + \v\big[(v\c u_0)(1+m_0)-v\c\n_x m_0\big]M + \v g\M,\\
u^\v(t,x)=&\ u_0+\v u,\\
\rho^\v(t,x)=&\ 1+h_0+\v\rho,
\end{aligned}
\end{equation} for some $(g,u,\rho)$ satisfying
\begin{equation}\label{Space of the solution}
\begin{aligned}
g \in & C^0([0,\infty);H^4_{x,v})\cap L^2([0,\infty);\mathcal{H}^5_{x,v}),\\
u \in & C^0([0,\infty);H^4_x)\cap L^2([0,\infty);H^5_x),\\
\rho \in & C^0([0,\infty);H^4_x)\cap L^2([0,\infty);H^4_x),
\end{aligned}
\end{equation}
provided that
\begin{equation}\label{Initial data}
\begin{aligned}
\|g^{in}\|_{H^4_{x,v}}^2+\|(u^{in},\rho^{in})\|^2_{H^4_x}+\|m_0^{in}\|_{H^6_x}^2+\|(u_0^{in}, h_0^{in})\|_{H^5_x}^2\leq \delta,
\end{aligned}
\end{equation}
where $\delta$ is independent of $\varepsilon$. In addition, the following energy estimate holds:
\begin{equation}\label{The global estimate}
\begin{aligned}
\sup_{t\geq0} \mathbb{E}(t) + C \int_0^\infty \mathbb{D}(s) \,ds \lesssim \mathbb{E}(0),
\end{aligned}
\end{equation}
for some constants $C > 0$ independent of $\v$, and the energy functional $\mathbb{E}$ and dissipation functional $\mathbb{D}$ are defined as follows:
\begin{equation}\label{instant energy functional}
\begin{aligned}
\mathbb{E}(t)=\|g\|_{H^4_{x,v}}^2+\|(u,\rho)\|^2_{H^4_x}+\|m_0\|_{H^6_x}^2+\|(u_0, h_0)\|_{H^5_x}^2
\end{aligned}
\end{equation}
and the associated dissipative functional
\begin{equation}\label{energy dissipative rate functional}
\begin{aligned}
\mathbb{D}(t)=& \frac{1}{\v^2} \left(\sum_{|\a|+|\b|\leq 4}\|\p_x^\a \n_v^\b(\I-\P)g\|_{\nu}^2+\|b-\v u\|_{H^4_x}^2\right) + \|(\n_x u,\up{div}_x u)\|_{H^4_x}^2+\|(\p_t u,\n_x\rho,\n_x a)\|^2_{H^3_x}\\[4pt]
& +\|(\n_x m_0,\n_x u_0)\|^2_{H^5_x}+\|(\p_t\n_x m_0,\p_t u_0,\n_x h_0,\up{div}_x u_0)\|_{H^4_x}^2.
\end{aligned}
\end{equation}
\end{theorem}

\begin{remark}
The smallness constant $\delta$ in (\ref{Initial data}) is usually larger than $\delta_0$ given in Proposition \ref{Solu-NSS}, and it will be determined in Section \ref{sec:energy}.
\end{remark}

By applying the Sobolev embedding $H^2 \hookrightarrow L^\infty$, we obtain the following corollary:
\begin{corollary}\label{Conergence}
Under the conditions of Theorem \ref{Main-Limits}, the following pointwise convergence holds:
\begin{equation}\label{Conv of t,x}
|f^\v(t,x,v) - (1+m_0(t,x))M(v)| + |u^\v(t,x) - u_0(t,x)| + |\rho^\v(t,x)- (1+h_0(t,x))|  \lesssim \v
\end{equation}
for $(t,x,v) \in \R^+ \times \R^3 \times \R^3$.
\end{corollary}

The rest of the paper is organized as follows: in Section \ref{sec:formal_analysis}, we present the formal analysis. In Section \ref{sec:energy}, we derive global-in-time \textit{a priori} estimates for the remainder system. Finally, the proofs of the main Theorem and Corollary are illustrated in Section \ref{sec:proof_of_main}.

\section{Formal analysis}
\label{sec:formal_analysis}

In this section, we formally derive the NSS equations \eqref{NSS}
\begin{equation}\label{Formal form}
\begin{aligned}
&f^\v =f_0+\v f_1 + \v^2 f_2 + \cdots,\\
&u^\v =u_0+\v u_1 + \v^2 u_2 + \cdots,\\
&\rho^\v =\rho_0 + \v \rho_1 + \v^2\rho_2 + \cdots.
\end{aligned}
\end{equation}

\textbf{Step 1:} Plugging the expansion \eqref{Formal form} into the VFP-CNS system $\eqref{VFP-CNS}_1$, we have
\begin{equation}\label{Formally form about f}
\begin{aligned}
&\p_t(f_0+\v f_1+\v^2 f_2+\cdots)+\frac{1}{\v}v\c\n_x(f_0+\v f_1+\v^2 f_2+\cdots )\\
= & \frac{1}{\v^2}\up{div}_v\Big\{\n_v(f_0+\v f_1+\v^2 f_2+\cdots)+\big[v-\v(u_0+\v u_1+\v^2 u_2+\cdots)\big] \cdot(f_0+\v f_1+\v^2 f_2+\cdots)\Big\}.
\end{aligned}
\end{equation}

The $\mathcal{O}(\frac{1}{\v^2})$ of \eqref{Formally form about f} reads
\begin{equation*}
    \up{div}_v(\n_v f_0+v f_0)=\up{div}_v[M\n_v(\frac{f_0}{M})]=0,
\end{equation*}
which implies that
\begin{equation}\label{macar-n0}
f_0(t,x,v)=n_0(t,x) M(v)
\end{equation}
with $n_0(t,x)$ to be determined.

The $\mathcal{O}(\frac{1}{\v})$ of \eqref{Formally form about f} reads
\begin{equation}\label{macar-1}
v\c\n_x f_0+\up{div}_v(u_0 f_0)=\up{div}_v[M\n_v(\frac{f_1}{M})].
\end{equation}
By substituting \eqref{macar-n0} into \eqref{macar-1} and further simplification, the equation \eqref{macar-1} becomes
\begin{equation*}
\up{div}_v\big[M(\n_v\frac{f_1}{M}+\n_x n_0-u_0 n_0)\big]=0.
\end{equation*}
Hence, $f_1$ can be solved as
\begin{equation}\label{Solu-f1}
f_1(t,x,v) = \big[(v\c u_0)n_0(t,x) - v\c\n_x n_0(t,x)\big]M(v).
\end{equation}

The $\mathcal{O}(1)$ of \eqref{Formally form about f} reads
\begin{equation*}
\p_t f_0+v\c\n_x f_1=\up{div}_v(\n_v f_2+vf_2-u_0f_1-u_1f_0),
\end{equation*}
which shows
\begin{equation}\label{Equa of n0}
\int_{\R^3} \p_t f_0+v\c\n_x f_1 dv=0.
\end{equation}

Substituting \eqref{macar-n0} and \eqref{Solu-f1} into \eqref{Equa of n0} gives
\begin{equation}\label{macro-n0}
\p_t n_0+\up{div}_x(n_0 u_0)=\Delta_x n_0.
\end{equation}

\textbf{Step 2:} Then, we substitute the expansions \eqref{Formal form} into the VFP-CNS system $\eqref{VFP-CNS}_2$ gives
\begin{equation}\label{Formal VFP-2}
\begin{aligned}
&(\rho_0+\v\rho_1+\cdots)\big[\p_t(u_0+\v u_1+\cdots)+(u_0+\v u_1+\cdots)\c\n_x(u_0+\v u_1+\cdots)\big]\\[3pt]
&+\n_x P(\rho_0+\v\rho_1+\cdots)+L(u_0+\v u_1+\cdots)\\[3pt]
=& \frac{1}{\v}\int_{\R^3}\big[v-\v(u_0+\v u_1+\cdots)\big](f_0+\v f_1+\cdots)dv,\\
\end{aligned}
\end{equation}

The $\mathcal{O}(\frac{1}{\v})$ of \eqref{Formal VFP-2} reads
\begin{equation*}
\int_{\R^3} vf_0 \, dv=0,
\end{equation*}
which is automatically correct since $f_0=n_0 M$ in \eqref{macar-n0}.

The $\mathcal{O}(1)$ of \eqref{Formal VFP-2} reads
\begin{equation*}
\begin{aligned}
\rho_0(\p_t u_0+u_0\c\n_x u_0)+\n_x P(\rho_0)+Lu_0&=\int_{\R^3}v f_1 \,dv -\int_{\R^3}u_0 f_0 \,dv.
\end{aligned}
\end{equation*}
This implies that
\begin{equation}\label{macro-u0}
\rho_0(\p_t u_0+u_0\c\n_x u_0)+\n_x P(\rho_0)+Lu_0+\n_x n_0=0,
\end{equation}
by considering the form of $f_0$ in \eqref{macar-n0} and $f_1$ in \eqref{Solu-f1}.

\textbf{Step 3:} Finally, substituting the expansions \eqref{Formal form} into the VFP-CNS system $\eqref{VFP-CNS}_3$ gives
\begin{equation}\label{Formal VFP-3}
\p_t(\rho_0+\v\rho_1+\cdots)+\up{div}_x\big[(\rho_0+\v\rho_1+\cdots)(u_0+\v u_1+\cdots)\big]=0.
\end{equation}

The $\mathcal{O}(1)$ of \eqref{Formal VFP-3} reads
\begin{equation}\label{macar-rho0}
\p_t\rho_0+\up{div}_x(\rho_0 u_0)=0.
\end{equation}

Combining the equations \eqref{macro-n0}, \eqref{macro-u0} and \eqref{macar-rho0}, it formally yields the limiting system \eqref{NSS}.

\section{Energy estimate}
\label{sec:energy}

\subsection{The remainder system}
Based on the formal analysis in Section \ref{sec:formal_analysis}, we introduce the remainder terms of the VFP-CNS system \eqref{VFP-CNS}:
\begin{equation}\label{Solu of special form-0}
g=\frac{f^\v-g_0-\v g_1\M}{\v \M}, \quad u=\frac{u^\v-u_0}{\v}, \quad \rho=\frac{\rho^\v-\v\rho_0}{\v},
\end{equation}
where 
\[
g_0=n_0M,\, \quad g_1=\big[(v\c u_0)n_0-v\c\n_x n_0\big]\M
\]
and $(n_0,u_0,\rho_0)$ is the solution to the NSS system \eqref{NSS}.
In other words, we seek a solution to the VFP-CNS system \eqref{VFP-CNS} in our newly-designed expansion form \eqref{Special form solution}.

By denoting $m_0=n_0-1,\, h_0=\rho_0-1$, the remainder system of $(g,u,\rho)$ is deduced as follows:
\begin{equation}\label{Remainder equations}
\left\{
\begin{aligned}
&\p_t g-\frac{1}{\v}v\c u\M+\frac{1}{\v^2}\L g=\frac{1}{\v}R_0+R_1,\\[3pt]
&(1+h_0+\v\rho)(\p_t u+u_0\c\n_x u+u\c\n_x u_0)+Lu-\frac{1}{\v}(b-\v u)=\frac{1}{\v}R_2+R_3,\\[3pt]
&\p_t \rho+\up{div}_x(h_0 u+\rho u_0)+\up{div}_x u+\v\up{div}_x(\rho u)=0,\\[3pt]
& (g, u, \rho)|_{t=0} = \big(g^{in}(x,v), u^{in}(x), \rho^{in}(x)\big) \to (0,0,0),\quad \up{as} \quad |x| \to +\infty,
\end{aligned}
\right.
\end{equation}
where
\begin{equation}\label{L-R0-R3}
\begin{aligned}
\L g=&-\frac{1}{\M}\up{div}_v[M\n_v(\frac{g}{\M})]=-\Delta_v g+\frac{|v|^2}{4}g-\frac{3}{2}g,\\[3pt]
Lu= & -\mu \Delta_x u-(\mu +\lambda)\n_x\up{div}_x u,\\[3pt]
g_1=&\big[(v\c u_0)(m_0+1)-v\c\n_x m_0\big]\M,\\[3pt]
R_0=&-(\I-\P_0)(v\c\n_x g_1)-v\c\n_x g-\frac{1}{\M}\up{div}_v(u_0 g_1\M+u_0 g\M+um_0 M),\\[3pt]
R_1=&-\p_t g_1-\frac{1}{\M}\up{div}_v(u g_1\M+ug\M),\\[3pt]
R_2=&-\big[\n_x P(1+h_0+\v\rho)-\n_x P(1+h_0)\big],\\[3pt]
R_3=&-\big[(\rho\p_t u_0+\rho u_0\c\n_x u_0+u m_0 +u_0 a)+\v(u\c\n_x u+h_0 u\c\n_x u+ua)+\v^2 \rho u\c\n_x u\big],
\end{aligned}
\end{equation}
and $(m_0,u_0,h_0)$ is the classical solution to
\begin{equation}\label{NSS-3}
\left\{
\begin{aligned}
&\p_t m_0+\up{div}_x\big[(1+m_0)u_0\big] = \Delta_x m_0,\\[3pt]
&\p_t\big[(1+h_0)u_0\big]+\up{div}_x\big[(1+h_0) u_0\otimes u_0\big]+\n_x P(1+h_0)+Lu_0+\n_x m_0=0,\\[3pt]
&\p_t h_0+\up{div}_x\big[(1+h_0)u_0\big]=0,\\[3pt]
&(m_0, u_0, h_0)|_{t=0}=(m^{in}_0(x), u^{in}_0(x), h^{in}_0(x)) \to (0,0,0), \quad \up{as} \quad |x|\to +\infty.
\end{aligned}
\right.
\end{equation}

We first present the local well-posedness of the remainder system \eqref{Remainder equations}.

\begin{proposition}\label{Local-in-time solution of the remainder equations}
Under the conditions of Proposition \ref{Solu-NSS}, assume that $\|g^{in}\|_{H^4_{x,v}}+\|(u^{in},\rho^{in})\|_{H^4_x}\leq M_0<\infty$, then for any given $0<\v \leq \v_0 := \v_0(M_0)$, there exists a $T_\v:=T_\v(M_0)>0$ such that the remainder system \eqref{Remainder equations} admits a unique solution $(g,u,\rho)$ satisfying
\begin{equation*}
\begin{aligned}
&g\in C^0([0,T_\v];H^4_{x,v})\cap L^2([0,T_\v];\mathcal{H}^5_{x,v}),\\
&u\in C^0([0,T_\v];H^4_x)\cap L^2([0,T_\v];H^5_x),\\
&\rho\in C^0([0,T_\v];H^4_x)\cap L^2([0,T_\v];H^4_x).
\end{aligned}
\end{equation*}
\end{proposition}

\begin{proof}
The proof is based on the standard contractive mapping argument.
We refer to \cite{LFCMYMWDH17, MNT83} for more details.
\end{proof}

\subsection{Energy and dissipation functionals}
\label{subsec:energy}

In this subsection, we present the energy and associated dissipation structure based on the Micro-Macro decomposition.

To state the energy estimate, we introduce the temporal energy and dissipation functionals of different parts:
\begin{itemize}
    \item Macroscopic energy functional of $(m_0, u_0, h_0)$: for $|\a_1| = 6$ and $|\beta_1|=1$,
    \begin{equation}\label{energy functionals}
    \begin{split}
     \E_{ma}(t) =&  \|\p_x^{\a_1} m_0 \|^2_{L^2_x} +\sum_{|\a|=1}^{5}K_{1,|\a|}\|(\p_x^\a h_0,\p_x^\a m_0,\p_x^\a u_0)\|^2_{L^2_x}+\sum_{|\a|=0}^{4}K_{2,|\a|}\|\p_x^\a\up{div}_x u_0\|^2_{L^2_x}\\[4pt]
&+\sum_{|\a|=1}^{5} K_{3,|\a|}\l\p_x^{\a}h_0,\p_x^{\a-\beta_1} u_0(1+h_0)^2\r_x+K_1 \Big[\|m_0\|^2_{L^2_x}+\|\sqrt{1+h_0}\, u_0\|^2_{L^2_x}\\[4pt]
&+\frac{2A}{\g-1}\|(1+h_0)^{\frac{\g}{2}}\|^2_{L^2_x}\Big],
    \end{split}
    \end{equation}
    and the associated macro dissipation functionals
    \begin{equation}\label{energy dissipation functionals}
    \D_{ma}(t) = \|(\n_x m_0,\n_x u_0)\|^2_{H^5_x} + \|(\p_t\n_x m_0,\p_t u_0,\n_x h_0,\up{div}_x u_0)\|_{H^4_x}^2,
    \end{equation}
    where $K_1$ and $K_{i,|\a|}, \, i=1,2,3,|\a|=1,\cdots,5$ are all positive constants. \\
    Note that one can verify that $\E_{ma}(t)\sim \|m_0\|_{H^6_x}^2+\|(u_0, h_0)\|_{H^5_x}^2$.\\

    \item Microscopic energy and dissipation functional of $(g,u,\rho)$: for $|\beta_1|=1$,
    \begin{equation}\label{part of energy functionals}
    \begin{aligned}
    \E_{mi,K,1}(t)=&\ \| g\|^2_{H^4_x L_v^2}  + \sum_{|\a|=0}^4\|\sqrt{1+h+\v\rho} \, \p_x^\a u\|^2+A\g\|\rho\|^2_{H^4_x},\\[4pt]
    \E_{mi,K,2}(t)=&\ \mu\|\n_x u\|^2_{H^3_x}+(\mu+\lambda)\|\up{div}_x u\|^2_{H_x^3},\\[4pt]
    \E_{mi,K,3}(t)=&\ \|\p_x\rho\|^2_{H^3_x} + \frac{2}{2\mu+\lambda}\sum_{|\a|=1}^4\l\p_x^\a\rho,(1+h+\v\rho)^2\p_x^{\a-\beta_1}u\r_x,\\[4pt]
    \E_{mi,K,4}(t)=&\ \sum_{0\leq|\a|+|\b|\leq 3}\bar{C}_{\a,\b}\|\p_x^\a\p_v^{\b+\beta_1}(\I-\P)g\|^2_{L^2_x},\\[4pt]
    \E_{mi,F}(t)=&\ \|(a,b)\|_{H^3_x}^2+\v\sum_{|\a|=0}^3\l\p_x^{\a+\beta_1}a,\p_x^\a b\r_x,
\end{aligned}
\end{equation}
and the associated micro dissipation functionals:
\begin{equation}\label{part of dissipation functionals}
\begin{aligned}
\D_{mi,K,1}(t)=&\ \frac{1}{\v^2}\left[\sum_{|\a|\leq 4}\|\p_x^\a(\I-\P)g\|_{\nu}^2+\|b-\v u\|_{H^4_x}^2\right] + \|(\n_x u,\up{div}_x u)\|_{H^4_x}^2,\\[4pt]
\D_{mi,K,2}(t)=&\ \|\p_t u\|_{H^3_x}^2,\\[4pt]
\D_{mi,K,3}(t)=&\ \|\n_x\rho\|^2_{H^3_x},\\[4pt]
\D_{mi,K,4}(t)=&\ \frac{1}{\v^2}\sum_{|\a|+|\b|\leq 3}\|\p_x^\a \p_v^{\b+\beta_1}(\I-\P)g\|_{\nu}^2,\\[4pt]
\D_{mi,F}(t)=&\ \|\n_x a\|^2_{H^3_x},
\end{aligned}
\end{equation}
with the constant $\bar{C}_{\a,\b} > 0$.

\item Total temporal energy and dissipation functionals:
\begin{equation}\label{energy functionals and energy dissipation functional}
\begin{aligned}
\E(t)=&\sum_{i=1}^4\lambda_i\E_{mi,K,i}(t)+\lambda_5\E_{mi,F}(t)+\lambda_6\E_{ma}(t),\\
\D(t)=&\sum_{i=1}^4\D_{mi,K,i}(t)+\D_{mi,F}(t)+\D_{ma}(t),
\end{aligned}
\end{equation}
where $\lambda_i> 0, \,1\leq i\leq 6$ are constants determined in \eqref{lambda constants}.
\end{itemize}

Considering the definitions of $\E(t),\,\D(t)$ in \eqref{energy functionals} and \eqref{energy dissipation functionals}, one can check that
\begin{equation}\label{equivalent}
\mathbb{E}(t)\sim \E(t), \quad \mathbb{D}(t)\sim\D(t),
\end{equation}
and also the energy functional $\mathbb{E}(t)$ is continuous for any $0\leq t\leq T$, i.e., there exists a positive constant $\Bar{C}$ depending on some known constants only, such that
\begin{equation*}
\begin{aligned}
\frac{1}{\Bar{C}}\E(t)\leq \mathbb{E}(t)\leq \Bar{C}\E(t), \quad \frac{1}{\Bar{C}}\D(t) \leq \mathbb{D}(t) \leq \Bar{C} \D(t).
\end{aligned}
\end{equation*}

We propose the following \textit{a priori} assumption: for any given $T>0$,
\begin{equation}\label{A priori assumption}
\sup_{0\leq t\leq T}\big[\|g\|_{H^4_{x,v}}^2+\|(u,\rho)\|^2_{H^4_x} \big] \leq \delta_1
\end{equation}
for a small positive constant $\delta_1$.

Now we are in a position to state our main energy estimate in Section \ref{sec:energy}.

\begin{proposition}\label{A priori estimates}
Under the assumptions of Theorem \ref{Main-Limits} and \eqref{A priori assumption}, let $(g,u,\rho)$ be the classical solutions to the remainder system \eqref{Remainder equations}, then, for any $0\leq t\leq T$,
\begin{equation}\label{total_energy_prop}
\E(t) + \tilde{C} \int_0^t \D(s) \, d s \leq \E(0),
\end{equation}
where the constants $\tilde{C} >0$ are independent of $\v$, $\delta$, $\delta_1$ and $T$.
\end{proposition}

The following Corollary can be deduced by Proposition \ref{A priori estimates} together with the equivalent condition \eqref{equivalent}.

\begin{corollary}\label{Instant energy-dissipative rate}
Under the assumptions of Theorem \ref{Main-Limits} and \eqref{A priori assumption}, let $(g,u,\rho)$ be the classical solutions to the remainder system \eqref{Remainder equations}, then, for any $t \in [0,T]$,
\begin{equation}\label{instant energy functional-energy dissipative rate functional}
\mathbb{E}(t) + \tilde{C}^*\int_0^t \mathbb{D}(s) \,ds \leq C^* \mathbb{E}(0),
\end{equation}
where the constants $\tilde{C}^*\,,C^* > 0$ are independent of $\v$, $\delta$, $\delta_1$ and $T$.
\end{corollary}

\subsection{Energy estimate of the remainder system}
\label{subsec:energy_estimate_remainder}

In this subsection, we will prove the Proposition \ref{A priori estimates}.
To this end, we first present the following lemma to evaluate $\frac{1}{1+h_0+\v\rho}$ and the singular term $\frac{1}{\v}\big[\n_x P(1+h_0+\v\rho)-\n_xP(1+h_0)\big]$ from pressure.
\begin{lemma}\label{estimate of Presure}
Under the assumptions of Proposition \ref{Solu-NSS} and \eqref{A priori assumption}, we have the following uniform estimates independent of the parameter $\varepsilon$,
\begin{equation}\label{estimate of 1+h+rho}
\| B^\a\|_{L^p_x}\lesssim
\left\{
\begin{array}{ccc}
&\|\n_x h_0\|_{H^{|\a|-1}_x}+\v\|\n_x\rho\|_{H^{|\a|-1}_x} \qquad  &\up{if}\, 1\leq|\a|\leq 4 \quad \up{and} \quad p=2,\\[4pt]
&\|\n_x h_0\|_{H^{|\a|}_x}+\v\|\n_x\rho\|_{H^{|\a|}_x} \qquad  &\up{if}\, 1\leq|\a|\leq 3 \quad \up{and} \quad 2<p\leq 6,\\[4pt]
&\|\n_x h_0\|_{H^{|\a|+1}_x}+\v\|\n_x\rho\|_{H^{|\a|+1}_x} \qquad  &\up{if} \,1\leq|\a|\leq 3 \quad \up{and} \quad  p=\infty,
\end{array}
\right.
\end{equation}
where $B^\a$ is denoted as
\begin{equation}
B^\a=\p_x^\a \left( \frac{1}{1+h_0+\v\rho} \right),
\end{equation}
and
\begin{equation}\label{estimate of 1+h0+rho}
\begin{aligned}
\|(1+h_0)^{\g-1}-1\|_{L^\infty_x}&\lesssim \|h_0\|_{H^2_x},\\[4pt]
\|(1+h_0)^{\g-1}-1\|_{H^1_x}&\lesssim \| h_0\|_{H^1_x},\\[4pt]
\|(1+h_0+\v\rho)^{\g-1}-(1+h_0)^{\g-1}\|_{L^\infty_x}&\lesssim \v\|\rho\|_{H^2_x},\\[4pt]
\|(1+h_0+\v\rho)^{\g-1}-(1+h_0)^{\g-1}\|_{H^1_x}&\lesssim\v\|\rho\|_{H^1_x}.
\end{aligned}
\end{equation}
Furthermore, we have
\begin{equation}\label{estimate of nable P-1}
\| B^\a_{\g}\|_{L^p_x}\lesssim
\left\{
\begin{array}{ccc}
\displaystyle \v\|\rho\|_{H^1_x}\|\n_x\rho\|_{L^2_x}& \qquad \up{if}\,\, p=2, \quad |\a| =0 \quad \up{and} \quad p=3,\,|\a|=0\\[4pt]
\displaystyle \v\|\rho\|_{H^2_x}\|\n_xh_0\|_{H^1_x}& \qquad \up{if}\,\, p=2,\,|\a|=1 \quad \up{and} \quad p=3,\,|\a|=1,\\[4pt]
\displaystyle \v\|\rho\|_{H^2_x}\|(\n_xh_0,\n_x\rho)\|_{H^1_x}& \qquad  \up{if} \,\, p=2,\,|\a|=2,\\[4pt]
\displaystyle \v\|\rho\|_{H^2_x}\|(\n_xh_0,\n_x\rho)\|_{H^2_x}& \qquad \up{if}\,\, p=3,\,|\a|=2 \quad \up{and} \quad p=2,\,|\a|=3,\\[4pt]
\displaystyle \v\|\rho\|_{H^3_x}\|(\n_xh_0,\n_x\rho)\|_{H^3_x}& \qquad \up{if}\,\, p=3,\,|\a|=3 \quad \up{and} \quad p=2,\,|\a|=4,\\
\end{array}
\right.
\end{equation}
where $B^\a_{\g}$ is denoted as
\begin{equation}\label{B-la}
B^\a_{\g} = \p_x^{\a} \left[(1+h_0+\v\rho)^{\g}\right] -\p_x^\a\left[(1+h_0)^{\g}\right] - \g\v(1+h_0+\v\rho)^{{\g}-1}\p_x^\a\rho.
\end{equation}
\end{lemma}

The complete proof can be found in Appendix \ref{sec:appdenix-pressure}, and we explain the key points in the following Remark \ref{re-lemma}.

\begin{remark}\label{re-lemma}

Note that handling the highly nonlinear pressure singularity relies fundamentally on the ``away from vacuum" assumption for the fluid density. A direct Taylor expansion cannot be utilized to estimate the pressure term because it yields an intractable remainder. Specifically, a direct expansion gives: 
\[
\begin{aligned} 
\|\n_xP(1+h_0+\v\rho)-\n_xP(1+h_0)\|_{H^s_x}=&\|\n_x[(1+h_0+\v\rho)^\g]-\n_x[(1+h_0)^\g]\|_{H^s_x}\\ =&\|\g\v\n_x\rho+\frac{\g(\g-1)}{2}\n_x\big[(1+h_0+\eta\rho)^{\g-2}\v^2\rho^2\big]\|_{H^s_x}, 
\end{aligned}
\]
where the intermediate state $\eta=\eta(t,x)\in(0,\v)$ depends on time $t$ and space $x$. Hence, the higher-order derivative term $\p_x^\a\eta(t,x)$ is hard to bound. 

To overcome this difficulty, our proof structure instead relies on a recursive exact algebraic expansion. We first explicitly calculate the spatial derivatives $\p_x^\a[(1+h_0+\v\rho)^\g]-\p_x^\a[(1+h_0)^\g]$ as shown in \eqref{nabla-x-1+h+rho} for each $0 \leq |\a| \leq 4$. Then, we can apply the Taylor expansion (strictly up to the second order) at the base level $(1+h_0)$ for each resulting term $(1+h_0+\v\rho)^{\tilde{\g}}-(1+h_0)^{\tilde{\g}}$ (see \eqref{estimate of bar-gamma-rho}). This decoupling approach isolates the singular components, bypassing the problematic $\eta(t,x)$ term entirely and allowing the singularities to be systematically absorbed by our adapted energy-dissipation structure.

\end{remark}

\subsubsection{Estimates of kinetic part of the remainder system}
\label{subsubsec:kinetic_remainder}

In this subsection, we will present the {\it a priori} estimates of the kinetic part of the remainder system \eqref{Remainder equations}-\eqref{NSS-3}, where the coercivity of $\L$ plays an essential role.

We start with the estimate of $\E_{mi,K,1}(t)$ and $\D_{mi,K,1}(t)$.
\begin{lemma}\label{estimate-kinetic-Mi-1}
Under the assumptions of Theorem \ref{Main-Limits} and \eqref{A priori assumption}, then, for any $0\leq t\leq T$,
\begin{equation}\label{Step one}
\begin{aligned}
\frac{1}{2}\frac{d}{dt}\E_{mi,K,1}(t) + C_1\D_{mi,K,1}(t) \leq \tilde{C}_1 \left(\D_{ma}(t)+\E^{\frac{1}{2}}(t)\D(t)\right),
\end{aligned}
\end{equation}
where the constants $\tilde{C}_1,\,C_1 > 0$ are independent of $\v$, $\delta$, $\delta_1$ and $T$.
\end{lemma}

\begin{proof}
By applying the derivative operator $\p_x^\a$ with $0\leq|\a|\leq 4$ to the remainder system \eqref{Remainder equations}, then multiplying each equation by $\p_x^\a g$, $\p_x^\a u$, $A\g\p_x^\a\rho$, respectively, and taking the integration with respect to $(x,v)$ and $x$, we have, for $|\beta|=1$,
\begin{multline}\label{Estimate of nabla-g-u-rho-0}
\frac{1}{2}\frac{d}{dt} \left(\|\p_x^\a g\|^2_{L^2_{x,v}} +\|\sqrt{1+h_0+\v\rho} \, \p_x^\a u\|^2_{L^2_x} +A\g\|\p_x^\a\rho\|^2_{L^2_x}\right) + \frac{c_0}{\v^2}\|\p_x^\a(\I-\P)g\|^2_{\nu}\\[4pt]
+\frac{1}{\v^2}\|\p_x^\a(b-\v u)\|^2_{L^2_x}
+\mu\|\p_x^{\a+\beta}u\|^2_{L^2_x}+(\mu+\lambda)\|\p_x^\a\up{div}_x u\|^2_{L^2_x}\\[4pt]
\leq\underbrace{\frac{1}{\v}\l\p_x^{\a} R_0,\p_x^\a g\r_{x,v}}_{B_{11}}+\underbrace{\l\p_x^\a R_1,\p_x^\a g\r_{x,v}}_{B_{12}}
+\underbrace{\frac{1}{\v}\l\p_x^\a R_2,\p_x^\a u\r_x}_{B_{13}}
+\underbrace{\l\p_x^\a R_3,\p_x^\a u\r_x}_{B_{14}}\\[4pt]
\underbrace{-\l\p_x^\a\big[(1+h_0)(u_0\c\p_xu+u\c\p_x u_0)\big],\p_x^\a u\r_x}_{B_{15}}
\underbrace{-\sum_{1 \leq |\tilde{\a}| \leq |\a|} C_{\a,\tilde{\a}} \l\p_x^{\tilde{\a}}(1+h_0+\v\rho) \p_t\p_x^{\a-\tilde{\a}}u_0,\p_x^\a u\r_x}_{B_{16}}\\[4pt]
\underbrace{-A\g\l\p_x^\a\up{div}_x(h_0u+\v\rho u), \p_x^\a\rho\r_x}_{B_{17}}
\underbrace{-A\g\l\p_x^\a\up{div}_x(\rho u_0),\p_x^\a\rho\r_x}_{B_{18}}\underbrace{-A\g\l\p_x^\a\up{div}_x u,\p_x^\a\rho\r_x}_{B_{19}},
\end{multline}
where we consider the decomposition $\L g=\L(\I-\P)g+\P_1g$ in \eqref{Decop-L} and inequality \eqref{Poin inequality}, and $R_0$, $R_1$, $R_2$, $R_3$ are given as in \eqref{L-R0-R3}.

Recalling the definition of $R_0$ in \eqref{L-R0-R3}, we can split $B_{11}$ into the following four terms:
\begin{equation}\label{B11-0}
\begin{aligned}
B_{11}=&\ \underbrace{-\frac{1}{\v}\l\p_x^\a\big[(\I-\P_0)(v\c\n_x g_1)\big],\p_x^\a g\r_{x,v}}_{B_{111}}\underbrace{-\frac{1}{\v}\l\p_x^\a\big[\frac{1}{\M}\up{div}_v(u_0 g_1 \M)\big],\p_x^\a g\r_{x,v}}_{B_{112}}\\
&\ \underbrace{-\frac{1}{\v}\l\p_x^\a\big[\frac{1}{\M}\up{div}_v(u_0 g \M)\big],\p_x^\a g\r_{x,v}}_{B_{113}}\underbrace{-\frac{1}{\v}\l\p_x^\a\big[\frac{1}{\M}\up{div}_v(u m_0 M)\big],\p_x^\a g\r_{x,v}}_{B_{114}}.
\end{aligned}
\end{equation}
For $B_{111}$, by substituting $g_1$ in \eqref{L-R0-R3} and noticing \eqref{self-adjoin of I-P}, we have, for $|\beta|=1$,
\begin{equation*}
\begin{aligned}
B_{111}=&\ -\frac{1}{\v}\l v\c\p_x^{\a+\beta}g_1,\p_x^\a(\I-\P)g\r_{x,v} - \frac{1}{\v}\l v\c\p_x^{\a+\beta}g_1,v\c\p_x^\a b\M\r_{x,v}\\
=&\ -\frac{1}{\v}\l v\c\p_x^{\a+\beta}(v\c u_0 m_0)\M,\p_x^\a(\I-\P)g\r_{x,v}-\frac{1}{\v}\l v\c\p_x^{\a+\beta}(v\c u_0)\M,\p_x^\a(\I-\P)g\r_{x,v}\\
&\ +\frac{1}{\v}\l v\c\p_x^{\a+\beta}(v\c\p_x m_0)\M,\p_x^\a(\I-\P)g\r_{x,v},
\end{aligned}
\end{equation*}
where $(\I-\P_0)g=(\I-\P_0)(\I-\P)g+v\c b\M=(\I-\P)g+v\c b\M$ is used in the second equality.\\
Then, $|B_{111}|$ is bounded by, for $|\beta| = 1$,
\begin{equation*}
\begin{aligned}
|B_{111}| \leq &\ \frac{C}{\v}\|u_0\|_{L^3_x}\|m_0\|_{L^6_x}\|(\I-\P)g\|_{L^2_{x,v}}+\frac{C}{\v}\| \n_x u_0\|_{L^2_x}\|(\I-\P)g\|_{L^2_
{x,v}} + \frac{C}{\v}\|\n_x m_0\|_{L^2_x}\|(\I-\P)g\|_{L^2_{x,v}}\\
&\leq \ \frac{C_0}{2^4\v^2}\|(\I-\P)g\|^2_{H^1_xL^2_v} + C\|(\n_x m_0,\n_x u_0)\|^2_{L^2_x} + \frac{C}{\v}\|u_0\|_{H^1_x}\|\n_x m_0\|_{L^2_x}\|(\I-\P)g\|_{L^2_{x,v}}\\
&\leq \ \frac{c_0}{2^4\v^2}\|(\I-\P)g\|^2_{H^1_xL^2_v}+C\D_{ma}(t)+C\E^{\frac{1}{2}}(t)\D(t), \quad \text{if} \quad |\alpha| = 0,\\[8pt]
|B_{111}| \leq &\ \frac{C}{\v}\sum_{1 \leq |\tilde{\a}|\leq |\a|}C_{\a,\tilde{\a}}\|\p_x^{\tilde{\a}}u_0\|_{L^4_x}\|\p_x^{\a-\tilde{\a}}m_0\|_{L^4_x} \|\p_x^\a(\I-\P)g\|_{L^2_{x,v}}\\
&\ +\frac{C}{\v}\| \p_x^\a u_0\|_{L^2_x}\|\p_x^\a(\I-\P)g\|_{L^2_{x,v}} + \frac{C}{\v}\|\p_x^{\a+\beta}m_0\|_{L^2_x}\|\p_x^\a(\I-\P)g\|_{L^2_{x,v}}\\
\leq& \ \frac{c_0}{2^4\v^2}\|\p_x^\a(\I-\P)g\|^2_{L^2_{x,v}}+C\D_{ma}(t)+C\E^{\frac{1}{2}}(t)\D(t), \quad \text{if} \quad |\alpha| \geq 1,
\end{aligned}
\end{equation*}
where the dissipative property \eqref{Poin inequality} is employed.\\
Hence, we obtain
\begin{equation*}\label{B111}
\begin{aligned}
|B_{111}| \leq &\ \frac{c_0}{2^4\v^2}\left[ \|(\I-\P)g\|^2_{H^1_xL^2_v} + \|\p_x^\a(\I-\P)g\|^2_{L^2_{x,v}}\right] + C\D_{ma}(t) + C\E^{\frac{1}{2}}(t)\D(t).
\end{aligned}
\end{equation*}
For $B_{112}$, by applying the similar argument as $B_{111}$, we also find 
\begin{equation*}\label{B112}
|B_{112}| \lesssim \ \E^{\frac{1}{2}}(t)\D(t).
\end{equation*}
For $B_{113}$, we have
\begin{equation*}\label{B113-0}
\begin{aligned}
B_{113}=&\ -\frac{1}{\v}\l\p_x^\a\big[u_0 \p_v(\I-\P)g-\frac{v}{2}\c u_0(\I-\P)g\big],\p_x^\a(\I-\P)g\r_{x,v}\\
&\ -\frac{1}{\v}\l\p_x^\a(u_0 \p_v \P g-\frac{v}{2}\c u_0\P g),\p_x^\a(\I-\P)g\r_{x,v}\\
&\ -\frac{1}{\v}\l\p_x^\a\big[u_0 \p_v(\I-\P)g-\frac{v}{2}\c u_0(\I-\P)g\big],\p_x^\a\P g\r_{x,v}\\
&\ -\frac{1}{\v}\l\p_x^\a(u_0 \p_v\P g-\frac{v}{2}\c u_0 \P g),\p_x^\a\P g\r_{x,v}\\
\leq&\ \frac{C}{\v}\left(\|u_0\|_{H^{|\a|}_x} + \|u_0\|_{H^2_x}\right) \left[\|(\p_x a,\p_x b)\|_{H^{|\a|-1}_x} + \sum_{0 \leq |\a|\leq 4}\|\p_x^\a(\I-\P)g\|_{\nu}\right] \sum_{0 \leq |\a|\leq 4}\|\p_x^\a(\I-\P)g\|_{\nu}\\
&\ -\frac{1}{\v}\l\p_x^\a(u_0 \p_v\P g-\frac{v}{2}\c u_0 \P g),\p_x^\a\P g\r_{x,v}\\
\leq&\ C\E^{\frac{1}{2}}(t)\D(t) \underbrace{+\frac{1}{\v}\l\p_x^\a(u_0 a),\p_x^\a(b-\v u)\r_x+ \l\p_x^\a(u_0 a),\p_x^\a u\r_x}_{B_{1131}},
\end{aligned}
\end{equation*}
and furthermore, notice the term $-au_0$ from $R_3$ in \eqref{L-R0-R3}, we have the following estimate for $B_{1131}$,   
\begin{equation*}
\begin{aligned}
|B_{1131}-\l u_0 a, u\r_x| \leq&\ \frac{1}{\v}\|u_0\|_{L^3_x}\|a\|_{L^6_x}\|b-\v u\|_{L^2_x}\\
\lesssim&\ \frac{1}{\v}\|u_0\|_{H^1_x}\|\n_x a\|_{L^2_x}\|b-\v u\|_{L^2_x}\\
\lesssim&\ \E^{\frac{1}{2}}(t)\D(t), \quad \text{if} \quad |\alpha| = 0 \\[8pt]
|B_{1131}| \leq&\ \frac{1}{\v}\|\p_x^\a u_0\|_{L^3_x}\|a\|_{L^6_x}\|\p_x^\a(b-\v u)\|_{L^2_x}+\frac{1}{\v}\|u_0\|_{L^\infty_x}\|\p_x^\a a\|_{L^2_x}\|\p_x^\a(b-\v u)\|_{L^2_x}\\
&\ +\frac{C}{\v}\sum_{1\leq |\tilde{\a}| \leq |\a|}\|\p_x^{\tilde{\a}}u_0\|_{L^4_x}\|\p_x^{\a-\tilde{\a}}a\|_{L^4_x}\|\p_x^\a(b-\v u)\|_{L^2_x}+\|\p_x^\a u_0\|_{L^3_x}\|a\|_{L^6_x}\|\p_x^\a u\|_{L^2_x}\\
&\ +\| u_0\|_{L^\infty_x}\|\p_x^\a a\|_{L^2_x}\|\p_x^\a u\|_{L^2_x}+C\sum_{1\leq|\tilde{\a}|\leq |\a|}\|\p_x^{\tilde{\a}}u_0\|_{L^4_x}\|\p_x^{\a-\tilde{\a}}a\|_{L^4_x}\|\p_x^\a u\|_{L^2_x}\\
\lesssim &\ \E^\frac{1}{2}(t)\D(t), \quad \text{if} \quad |\alpha| \geq 1.
\end{aligned}
\end{equation*}
Hence, we have
\begin{equation*}\label{B113}
|B_{113}| \lesssim \ \E^\frac{1}{2}(t)\D(t).
\end{equation*}
By applying the similar argument as $B_{113}$, we also find
\begin{equation*}\label{B114}
|B_{114}| \lesssim \ \E^\frac{1}{2}(t)\D(t).
\end{equation*}
By inserting all the estimates of $B_{111}$, $B_{112}$, $B_{113}$, $B_{114}$ into \eqref{B11-0}, we obtain that
\begin{equation}\label{B11}
\begin{aligned}
|B_{11}| \leq\frac{c_0}{2^4\v^2} \left[\|(\I-\P)g\|_{H^1_x L^2_v}^2+\|\p_x^\a(\I-\P)g\|^2_{L^2_{x,v}}\right] + C\left(\D_{ma}(t)+\E^{\frac{1}{2}}(t)\D(t)\right).
\end{aligned}
\end{equation}

For $B_{12}$, by recalling the definition of $R_1$ in \eqref{L-R0-R3}, we can split it into
\begin{equation}\label{B12-0}
B_{12}=\ \underbrace{-\l\p_x^\a \p_t g_1, \p_x^\a g\r_{x,v}}_{B_{121}} \underbrace{-\l\p _x^\a\big[\frac{1}{\M}\up{div}_v(ug_1\M+ug\M)\big],\p_x^\a g\r_{x,v}}_{B_{122}}.
\end{equation}
Furthermore, by substituting $g_1$, $B_{121}$ is divided into the following two parts:
\begin{equation}\label{B121-0}
B_{121}= \underbrace{-\l\p_t \p_x^\a(m_0u_0),\p^\a_x b\r_x}_{B_{1211}}\underbrace{-\l \p_t\p_x^\a u_0+\p_t\p_x^{\a+1} m_0,\p^\a_x b\r_x}_{B_{1212}}.
\end{equation}
For $B_{1211}$,
\begin{equation*}\label{B1211-1}
\begin{aligned}
|B_{1211}| =&\ \left|-\l\p_tm_0 u_0+m_0\p_t u_0,b-\v u\r_x-\v\l\p_tm_0 u_0+m_0\p_t u_0, u\r_x \right|\\[4pt]
\leq&\ \left(\|\p_tm_0\|_{L^2_x}\|u_0\|_{L^\infty_x}+\|\p_t u_0\|_{L^2_x}\| m_0\|_{L^\infty_x}\right) \|b-\v u\|_{L^2_x}\\[4pt]
&\qquad \qquad \qquad \qquad+\v\left(\|\p_tm_0\|_{L^2_x}\|u_0\|_{L^3_x}+\|\p_t u_0\|_{L^2_x}\| m_0\|_{L^3_x}\right)\|u\|_{L^6_x}\\[4pt]
\lesssim&\ \big(\|\p_tm_0\|\|u_0\|_{H^2_x}+\|\p_t u_0\|_{L^2_x}\| m_0\|_{H^2_x}\big)\|b-\v u\|_{L^2_x}\\[4pt]
&\qquad \qquad \qquad \qquad+\v\big(\|\p_tm_0\|_{L^2_x}\|u_0\|_{H^1_x}+\|\p_t u_0\|_{L^2_x}\| m_0\|_{H^1_x}\big)\|\p_x u\|_{L^2_x}\\
\lesssim&\ \E^{\frac{1}{2}}(t)\D(t), \quad \text{if} \quad |\alpha| = 0, \\[8pt]
|B_{1211}|=&\ \left|-\l\p_t \p_x^\a(m_0u_0),\p^\a_x(b-\v u)\r_x-\v\l\p_t \p_x^\a(m_0u_0),\p^\a_x u\r_x \right|\\[4pt]
\lesssim& \  \left( \|\p_tu_0\|_{H^{|\a|}_x}\|m_0\|_{H^{|\a|}_x}+\|u_0\|_{H^{|\a|}_x}\|\p_t m_0\|_{H^{|\a|}_x} \right)\|\p_x^\a(b-\v u)\|_{L^2_x}\\[4pt]
&\ +\v\left(\|\p_t u_0\|_{H^{|\a|}_x}\| m_0\|_{H^{|\a|}_x}\|\p_x u\|_{H^{|\a|-1}_x}  +\|\p_t m_0\|_{H^{|\a|}_x}\| u_0\|_{H^{|\a|}_x}\|\p_x u\|_{H^{|\a|-1}_x}\right)\\[4pt]
\lesssim&\ \E^{\frac{1}{2}}(t)\D(t), \quad \text{if} \quad |\alpha|\geq 1,
\end{aligned}
\end{equation*}
which then implies that
\begin{equation*}\label{B1211}
B_{1211}\lesssim \ \E^{\frac{1}{2}}(t)\D(t),
\end{equation*}
 while for $B_{1212}$, we have
\begin{equation*}
\begin{aligned}
|B_{1212}|=&\ \left| -\l\p_t u_0+\p_t\n_x m_0,b-\v u\r_x-\v\l\p_t u_0+\p_t\n_x m_0, u\r_x \right|\\
\leq&\ \frac{1}{2^4\v^2}\|b-\v u\|^2_{L^2_x}+C\|\p_t u_0+\p_t\n_x m_0\|^2_{L^2_x} + \v\|\p_t u_0+\p_t\n_x m_0\|_{\dot{H}^{-1}_x}\| u\|_{\dot{H}^1_x}\\
\leq&\ \frac{1}{2^4\v^2}\|b-\v u\|^2_{L^2_x}+C\|(\p_t u_0,\p_t\n_x m_0)\|^2_{L^2_x}+\frac{\mu}{2^4}\|\n_x u\|_{L^2_x}^2\\
\leq&\ C\D_{ma}(t)+\frac{1}{2^4\v^2}\|b-\v u\|^2_{L^2_x}+\frac{\mu}{2^4}\|\n_x u\|^2_{L^2_x}, \quad \text{if} \quad |\alpha| = 0,\\[8pt]
|B_{1212}| =&\  \left|-\l\p_t \p_x^\a u_0+\p_t\p_x^{\a+\beta}m_0,\p_x^\a(b-\v u)\r_x-\v\l\p_t\p_x^\a u_0+\p_t\p_x^{\a+\beta}m_0, \p_x^\a u\r_x \right|\\[4pt]
\leq&\ \frac{1}{2^4\v^2}\|\p_x^\a(b-\v u)\|^2_{L^2_x}+C\|(\p_t\p_x^\a u_0,\p_t\p_x^{\a+\beta}m_0)\|^2_{L^2_x}+\frac{\mu}{2^4}\|\p_x^{\a}u\|^2_{L^2_x}\\[4pt]
\leq&\ C\D_{ma}(t)+\frac{1}{2^4\v^2}\|\p_x^\a(b-\v u)\|^2_{L^2_x}+\frac{\mu}{2^4}\|\p_x^{\a}u\|^2_{L^2_x}, \quad \text{if} \quad |\alpha| \geq 1,
\end{aligned}
\end{equation*}

which leads to
\begin{equation*}\label{B1212}
|B_{1212}| \leq\ C\D_{ma}(t)+\frac{1}{2^4\v^2}\|\p_x^\a(b-\v u)\|^2_{L^2_x}+\frac{\mu}{2^4}\|(\n_x u,\p_x^{\a}u)\|^2_{L^2_x}.
\end{equation*}
By substituting the estimates of $B_{1211}$ and $B_{1212}$ into \eqref{B121-0}, we obtain
\begin{equation}\label{B121}
|B_{121}| \leq \ C\D_{ma}(t)+\frac{1}{2^4\v^2}\|\n_x^\a(b-\v u)\|^2_{L^2_x}+\frac{\mu}{2^4}\|(\n_x u,\p_x^{\a}u)\|^2_{L^2_x}+C\E^{\frac{1}{2}}(t)\D(t).
\end{equation}
We can also apply the similar argument as for $B_{113}$, and find $B_{122}$ to be bounded by
\begin{equation}\label{B122}
\begin{aligned}
|B_{122}| =\ \left| -\l\p _x^\a\big[\frac{1}{\M}\up{div}_v(ug_1\M+ug\M)\big],\p_x^\a g\r_{x,v} \right|
\lesssim \E^{\frac{1}{2}}(t)\D(t).
\end{aligned}
\end{equation}
Combining \eqref{B12-0}, \eqref{B121}, and \eqref{B122}, we have
\begin{equation}\label{B12}
|B_{12}| \leq\ \frac{1}{2^4\v^2}\|\p_x^\a(b-\v u)\|^2_{L^2_x}+\frac{\mu}{2^4}\|(\n_x u,\p_x^{\a}u)\|^2_{L^2_x}+ C \left(\D_{ma}(t)+\E^{\frac{1}{2}}(t)\D(t)\right).
\end{equation}

Recalling the definition of $R_2$ in \eqref{L-R0-R3}, we can divided $B_{13}$ as follows: for $|\beta| = 1$,
\begin{equation*}
\begin{aligned}
B_{13}=&\ \frac{A}{\v}\l\p_x^{\a+\beta}\big[(1+h_0)^\g\big]-\p_x^{\a+\beta}\big[(1+h_0+\v\rho)^\g\big],\p_x^\a u\r_x\\[4pt]
=&\frac{A}{\v}\l\p_x^{\a}\big[(1+h_0+\v\rho)^\g\big]-\p_x^{\a}\big[(1+h_0)^\g\big],\p_x^\a\up{div}_x u\r_x\\[4pt]
=&\frac{A}{\v}\l B_{\g}^{\a},\p_x^\a\up{div}_x u\r+A\g\l\big[(1+h_0+\v\rho)^{\g-1}-(1+h_0)^{\g-1}\big]\p^\a_x\rho,\p_x^\a\up{div}_x u\r_x\\[4pt]
&+A\g\l\big[(1+h_0)^{\g-1}-1\big]\p^\a_x\rho,\p_x^\a\up{div}_x u\r_x - B_{19},
\end{aligned}
\end{equation*}
where $B_{\g}^{\a}$ is defined as in \eqref{B-la} and $B_{19}$ has been given in \eqref{Estimate of nabla-g-u-rho-0}.\\
By applying the Lemma \ref{estimate of Presure}, we can have the following estimate of $B_{13}$, for $|\beta| = 1$,
\begin{equation*}
\begin{aligned}
|B_{13} + B_{19}|  \lesssim &\ \frac{1}{\v}\|B_{\g}^0\|_{L^2_x}\|\n_x u\|_{L^2_x}+\|(1+h_0+\v\rho)^{\g-1}-(1+h_0)^{\g-1}\|_{L^{3}_x}\|\rho\|_{L^6_x}\|\n_x u\|_{L^2_x}\\
&\ +\|(1+h_0)^{\g-1}-1\|_{L^3_x}\|\rho\|_{L^6_x}\|\n_x u\|_{L^2_x}\\[4pt]
\lesssim &\ \|\rho\|_{H^1_x}\|\n_x\rho\|_{L^2_x}\|\n_x u\|_{L^2_x}+\v\|\rho\|_{H^1_x}\|\n_x\rho\|_{L^2_x}\|\n_x u\|_{L^2_x}+\|h_0\|_{H^1_x}\|\n_x\rho\|_{L^2_x}\|\n_x u\|_{L^2_x} \\[4pt]
\lesssim&\ \E^{\frac{1}{2}}(t)\D(t), \quad \text{if} \quad |\alpha| = 0\\[8pt]
|B_{13} + B_{19}| \lesssim &\ \frac{1}{\v}\|B_{\g}^\a\|_{L^2_x}\|\p_x^{\a+\beta} u\|_{L^2_x}+\|(1+h_0+\v\rho)^{\g-1}-(1+h_0)^{\g-1}\|_{L^\infty_x}\|\p_x^\a\rho\|_{L^2_x}\|\p_x^{\a+\beta} u\|_{L^2_x}\\[4pt]
&\ +\|(1+h_0)^{\g-1}-1\|_{L^\infty_x}\|\p_x^\a\rho\|_{L^2_x}\|\p_x^{\a+\beta} u\|_{L^2_x}\\[4pt]
\lesssim &\ \|\rho\|_{H^2_x}\|(\n_x\rho,\n_x h_0)\|_{H^3_x}\|\p_x^{\a+\beta} u\|_{L^2_x}+\v\|\rho\|_{H^2_x}\|\n_x\rho\|_{L^2_x}\|\p_x^{\a+\beta} u\|_{L^2_x}\\[4pt]
&\ +\|h_0\|_{H^2_x}\|\n_x\rho\|_{L^2_x}\|\p_x^{\a+\beta} u\|_{L^2_x}\\[4pt]
\lesssim &\ \E^{\frac{1}{2}}(t)\D(t), \quad \text{if} \quad |\alpha| \geq 1.
\end{aligned}
\end{equation*}
Therefore, we have
\begin{equation}\label{B13-B19}
|B_{13}+B_{19}| \lesssim \ \E^{\frac{1}{2}}(t)\D(t).
\end{equation}
By further noticing the definition of $R_3$ in \eqref{L-R0-R3} and considering \eqref{Key inequality},
we can apply the similar argument of estimating $B_{11}$ into $B_{14}$, $B_{15}$, $B_{17}$, and $B_{18}$ such that
\begin{equation}\label{B14-B15-B17-B18}
|B_{14}| + |B_{15}| + |B_{17}| + |B_{18}| \lesssim \ \E^{\frac{1}{2}}(t)\D(t),
\end{equation}

On the other hand, for rest term $B_{16}$, we have
\begin{equation*}
B_{16}= \ -\sum_{1\leq |\tilde{\a}| \leq |\a|}C_{\a, \tilde{\a}}\l\p_x^{\tilde{\a}}h_0 \, \p_t\p_x^{\a-\tilde{\a}}u_0, \p_x^\a u_0\r_x -\v\sum_{1 \leq|\tilde{\a}| \leq |\a|} C_{\a,\tilde{\a}}\l\p_x^{\tilde{\a}}\rho \, \p_t\p_x^{\a-\tilde{\a}}u_0,\p_x^\a u_0\r_x,
\end{equation*}
from which, we find that $B_{16}$ only exists when $|\a|\geq 1$, and can be bounded by
\begin{equation*}
\begin{aligned}
|B_{16}| \lesssim&\ \|\p_x^\a h_0\|_{L^2_x}\|\p_t u_0\|_{H^2_x}\|\p_x^\a u\|_{L^2_x}+\sum_{1\leq |\tilde{\a}| \leq |\a|-1}\|\p_x^{\tilde{\a}}h_0\|_{L^4_x}\|\p_t\p_x^{\a-\tilde{\a}}u_0\|_{L^4_x}\|\p_x^\a u_0\|_{L^2_x}\\
\lesssim&\ \v\|\p_x^\a \rho\|_{L^2_x} \|\p_t u_0\|_{H^2_x}\|\p_x^\a u\|_{L^2_x}+ C\v\sum_{1\leq|\tilde{\a}| \leq |\a|-1}\|\p_x^{\tilde{\a}}\rho\|_{L^4_x}\|\p_t\p_x^{\a-\tilde{\a}}u_0\|_{L^4_x}\|\p_x^\a u_0\|_{L^2_x}\\
\lesssim&\ \E^{\frac{1}{2}}(t)\D(t).
\end{aligned}
\end{equation*}
Therefore, we deduce that
\begin{equation}\label{B16}
|B_{16}| \lesssim \ \E^{\frac{1}{2}}(t)\D(t).
\end{equation}

Putting all the estimates \eqref{B11}, \eqref{B12}, \eqref{B13-B19}, \eqref{B14-B15-B17-B18} and \eqref{B16} into \eqref{Estimate of nabla-g-u-rho-0}, the proof can be finally completed by summing up for all $0 \leq |\a| \leq 4$.

\end{proof}

To estimate $\p_t u$, we give the following lemma about $\E_{mi,K,2}(t)$ and $\D_{mi,K,2}(t)$ defined in \eqref{part of energy functionals} and \eqref{part of dissipation functionals}.

\begin{lemma}\label{estimate-kinetic-Mi-2}
Under the assumptions of Theorem \ref{Main-Limits} and \eqref{A priori assumption}, then, for any $0\leq t\leq T$,
\begin{equation}\label{Step two}
\begin{aligned}
\frac{1}{2}\frac{d}{dt}\E_{mi,K,2}(t)+C_2\D_{mi,K,2}(t)\leq \tilde{C}_2\left(\D_{mi,K,1}(t)+\D_{mi,K,3}(t)+\E^{\frac{1}{2}}(t)\D(t)\right),
\end{aligned}
\end{equation}
where the constants $\tilde{C}_2,\,C_2> 0$ are independent of $\v$, $\delta$, $\delta_1$ and $T$.
\end{lemma}

\begin{proof}
By applying $\p_x^\a$ to both-hand-sides of the remainder system $\eqref{Remainder equations}_2$ with $0\leq|\a|\leq 4$, multiplying it by $\p_t\p_x^\a u$ and integrating over $\R^3$, we have, for $|\beta|=1$,

\begin{equation}\label{Estimate of partial-nabla-u-1}
\begin{aligned}
&\frac{1}{2}\frac{d}{dt} \left[\mu\|\p_x^{\a+\beta}u\|^2_{L^2_x}+(\mu+\lambda)\|\p_x^\a\up{div}_x u\|^2_{L^2_x}\right] + \|\sqrt{1+h_0+\v\rho} \, \p_t\p_x^\a u\|^2_{L^2_x}\\
&=\underbrace{-\frac{1}{\v}\l\p_x^\a(b-\v u),\p_t\p_x^\a u\r_x}_{B_{21}}
\underbrace{-\sum_{1 \leq |\tilde{\a}| \leq |\a|}C_{\a,\tilde{\a}}\l(\p_x^{\tilde{\a}}h+\v\p_x^{\tilde{\a}}\rho) \, \p_t\p_x^{\a-\tilde{\a}}u, \, \p_t\p_x^\a u\r_x}_{B_{22}}\\
&\quad \underbrace{-\l\p_x^\a\left[(1+h_0+\v\rho)(u\c\n_x u_0 + u_0\c\n_x u)\right],\p_t\p_x^\a u\r_x}_{B_{23}}+\underbrace{\frac{1}{\v}\l\p_x^\a R_2,\p_t\p_x^\a u\r_x}_{B_{24}}+\underbrace{\l\p_x^\a R_3,\p_t\p_x^\a u\r_x}_{B_{25}}.
\end{aligned}
\end{equation}

For $B_{21}$, we have
\begin{equation}\label{B21}
|B_{21}| \leq \frac{1}{4^3}\|\p_t\p_x^\a u\|^2_{L^2_x}+\frac{C}{\v^2}\|\p_x^\a(b-\v u)\|^2_{L^2_x} \leq \frac{1}{4^3}\|\p_t\p_x^\a u\|^2_{L^2_x}+C\D_{mi,K,1}(t),
\end{equation}
where $\D_{mi,K,1}(t)$ is defined in \eqref{part of energy functionals}.

For $B_{22}$,
\begin{equation}\label{B22}
\begin{aligned}
|B_{22}| \lesssim & \sum_{1 \leq |\tilde{\a}| \leq |\a|} \left(\|\p_x^{\tilde{\a}}h_0\|_{L^4_x}+\v\|\p_x^{\tilde{\a}}\rho\|_{L^4_x}\right) \|\p_t\p_x^{\a-\tilde{\a}}u\|_{L^4_x}\|\p_t\p_x^\a u\|_{L^2_x}\\
\lesssim & \sum_{1\leq|\tilde{\a}| \leq |\a|}\left(\|\p_x^{\tilde{\a}}h_0\|_{H^1_x}+\v\|\p_x^{\tilde{\a}}\rho\|_{H^1_x}\right) \|\p_t\p_x^{\a-\tilde{\a}}u\|_{L^4_x}\|\p_t\p_x^\a u\|_{L^2_x}\\
\lesssim & \ \E^{\frac{1}{2}}(t)\D(t),
\end{aligned}
\end{equation}
where the H$\ddot{\up{o}}$lder inequality and Sobolev embedding are used.

We obtain the estimate for $B_{23}$ and $B_{25}$ by using the similar argument as for $B_{22}$,,
\begin{equation}\label{B23-B25}
\begin{aligned}
|B_{23}|+|B_{25}| \lesssim \ \E^{\frac{1}{2}}(t)\D(t)
\end{aligned}
\end{equation}

Recalling the definition of $R_2$ in \eqref{L-R0-R3}, and applying \eqref{Key inequality} and Lemma \ref{estimate of Presure}, we can estimate $B_{24}$, for $|\beta|=1$,
\begin{equation}\label{B24}
\begin{aligned}
|B_{24}| =& -\frac{A}{\v}\l\p_x^{\a+\beta}\big[(1+h_0+\v\rho)^\g\big]-\p_x^{\a+\beta}\big[(1+h)^\g\big],\p_t\p_x^\a u\r_x\\[4pt]
=&-\frac{A}{\v}\l\p_x^{\a+1}\big[(1+h_0+\v\rho)^\g\big]-\p_x^{\a+1}\big[(1+h_0)^\g\big]-\g\v (1+h_0+\v\rho)^{\g-1}\p_x^{\a+\beta}\rho,\p_t\p_x^\a u\r_x\\[4pt]
&-A\g\l\big[(1+h_0+\v\rho)^{\g-1}-(1+h_0)^{\g-1}\big]\p^{\a+\beta}_x\rho,\p_t\p_x^\a u\r_x\\[4pt]
&-A\g\l\big[(1+h_0)^{\g-1}-1\big]\p^{\a+\beta}_x\rho,\p_t\p_x^\a u\r_x-A\g\l\p^{\a+\beta}_x\rho,\p_t\p_x^\a u\r_x\\[4pt]
\leq&\frac{1}{4^3}\|\p_t\p_x^\a u\|^2_{L^2_x}+C\|\p_x^{\a+\beta}\rho\|^2_{L^2_x}+C\E^{\frac{1}{2}}(t)\D(t)\\[4pt]
\leq&\frac{1}{4^3}\|\p_t\p_x^\a u\|^2_{L^2_x} + C\big(\D_{mi,K,3}(t) + \E^{\frac{1}{2}}(t)\D(t)\big),
\end{aligned}
\end{equation}
where $\D_{mi,K,3}(t)$ is defined in \eqref{part of energy functionals}.

Combining with \eqref{Estimate of partial-nabla-u-1}, \eqref{B21}, \eqref{B22}, \eqref{B23-B25}, and \eqref{B24} as well as considering \eqref{estimate of rho0-rho-k}, we have, for $0\leq|\a|\leq 4$ and $|\beta|=1$,
\begin{multline}\label{Estimate of partial-nabla-u}
\frac{1}{2}\frac{d}{dt}\big[\mu\|\p_x^{\a+\beta}u\|^2_{L^2_x}+(\mu+\lambda)\|\p_x^\a\up{div}_x u\|^2_{L^2_x}\big]+\frac{1}{4}\|\p_t\p_x^\a u\|^2_{L^2_x}\\
\lesssim \ \left(\D_{mi,K,1}(t)+\D_{mi,K,3}(t)+\E^{\frac{1}{2}}(t)\D(t)\right).
\end{multline}
Summing up the above inequality with $0 \leq |\a| \leq 4$, the proof can be completed.
\end{proof}

To close the energy estimate of the remainder system \eqref{Remainder equations}-\eqref{NSS-3}, we also need to study the dissipation of $\rho$, i.e., the estimate of $\E_{mi,K,3}(t)$ and $\D_{mi,K,3}(t)$.

\begin{lemma}\label{estimate-kinetic-Mi-3}
Under the assumptions of Theorem \ref{Main-Limits} and \eqref{A priori assumption}, then, for any $0\leq t\leq T$,
\begin{equation}\label{Step three}
\frac{1}{2}\frac{d}{dt}\E_{mi,K,3}(t)+C_3\D_{mi,K,3}(t) \leq \tilde{C}_3\left( \D_{mi,K,1}(t)+\E^{\frac{1}{2}}(t)\D(t)\right),
\end{equation}
where the constants $\tilde{C}_3,\,C_3 > 0$ are independent of $\v$, $\delta$, $\delta_1$ and $T$.
\end{lemma}

\begin{proof}
We start with the following calculation: for $1 \leq |\a| \leq 4$ and $|\beta|=1$,
\begin{equation}\label{B3-0}
\begin{aligned}
&\frac{d}{dt}\int \p_x^\a\rho\c\left[\frac{\p_x^\a \rho}{2} + \frac{(1+h_0+\v\rho)^2}{2\mu+\lambda}\p_x^{\a-\beta}u \right] dx\\[4pt]
=&\underbrace{\l\p_x^\a\rho,\p_t\p_x^\a \rho\r_x}_{B_{31}} + \underbrace{\frac{1}{2\mu+\lambda}\l(1+h_0+\v\rho)^2\p_t\p_x^\a\rho,\p_x^{\a-\beta}u_0 \r_x}_{B_{32}}\\
&+\underbrace{\frac{1}{2\mu+\lambda}\l(1+h_0+\v\rho)^2\p_x^\a\rho,\p_t\p_x^{\a-\beta}u \r_x}_{B_{33}}
+\underbrace{\frac{2}{2\mu+\lambda}\l(1+h_0+\v\rho)(\p_t h_0+\v\p_t\rho)\p_x^\a\rho,\p_x^{\a-\beta}u \r_x}_{B_{34}}.
\end{aligned}
\end{equation}

For $B_{31}$, by substituting $\eqref{Remainder equations}_3$ and considering \eqref{A priori assumption}, we have, for $|\beta|=1$,
\begin{equation}\label{B31}
\begin{aligned}
B_{31}=&-\l\p_x^\a\rho, \p_x^\a\up{div}_x[(1+h_0+\v\rho)u]\r_x-\l\p_x^\a\rho, \p_x^\a\up{div}_x(\rho u_0)\r_x\\[4pt]
=&-\l\p_x^\a\rho,\p_x^\a\up{div}_x u(1+h_0+\v\rho)\r_x-\l\p_x^\a\rho,(\p_x^\a h_0+\v\p_x^\a\rho)\up{div}_x u\r_x+\frac{\v}{2}\l\up{div}_x u, (\p_x^\a\rho)^2\r_x\\[4pt]
&-\l\p_x^\a\rho,\p_x^\a u (\p_xh_0+\v\p_x\rho)\r_x-\l\p_x^\a\rho  u,\p_x^{\a+\beta}h_0\r_x-C_{\a,1} \l\p_x^\a\rho \p_x u,\p_x^\a h_0+\v\p_x^\a\rho\r_x\\[4pt]
&- C_{\a,1}\l\p_x^\a\rho,\p_x^{\a-\beta}\up{div}_x u\c(\p_x h_0+\v\p_x\rho)\r_x\\[4pt]
&-\sum_{1 \leq |\tilde{\a}| \leq |\a| -2}C_{\a,\tilde{\a}} \l\p_x^\a\rho,\p_x^{\tilde{\a}}\up{div}_x u (\p_x^{\a-\tilde{\a}}h_0+\v\p_x^{\a-\tilde{\a}}\rho)\r_x\\[4pt]
&-\sum_{1 \leq |\tilde{\a}| \leq |\a|-1}C_{\a,\tilde{\a}}\l\p_x^\a\rho,\p_x^{\tilde{\a}} u (\p_x^{\a+\beta -\tilde{\a}}h_0 + \v\p_x^{\a+\beta-\tilde{\a}}\rho)\r_x\\[4pt]
\leq& -\l\p_x^\a\rho,\p_x^\a\up{div}_x u(1+h_0+\v\rho)\r_x+C\E^{\frac{1}{2}}(t)\D(t).
\end{aligned}
\end{equation}

Similar to the estimate of $B_{31}$, for $B_{34}$, we have
\begin{equation}\label{B34}
B_{34}=-\l\up{div}_x[(1+h_0+\v\rho)u]\p_x^\a\rho,\p_x^{\a-1}u\r_x-\l\up{div}_x(\rho u_0)\p_x^\a\rho, \p_x^{\a-1}u\r_x
\lesssim \ \E^{\frac{1}{2}}(t)\D(t).
\end{equation}

For $B_{32}$,
\begin{equation*}
\begin{aligned}
B_{32}=&\underbrace{\frac{1}{2\mu+\lambda}\l\p_x^\a\big[u(1+h_0+\v\rho)\big],\p^\a_x u(1+h_0+\v\rho)^2 \r_x}_{B_{321}}\\
&\qquad \qquad \qquad +\underbrace{\frac{2}{2\mu+\lambda}\l\p_x^\a[u(1+h_0+\v\rho)],\p_x^{\a-\beta}u (\p_x h_0+\v\p_x\rho)(1+h_0+\v\rho)\r_x}_{B_{322}},
\end{aligned}
\end{equation*}
with $|\beta| = 1$.\\
Furthermore, $B_{321}$ is bounded by
\begin{equation*}
\begin{aligned}
B_{321}
\lesssim&\ \|\p_x u\|_{H^{|\a|-1}_x}^2+C\|u\|_{H^2_x}\|\p_x^\a h_0+\v\p_x^\a \rho\|_{L^2_x}\|\p_x^\a\rho\|_{L^2_x}\\[4pt]
&\ \qquad \qquad \qquad +\sum_{1\leq |\tilde{\a}| \leq |\a|-1}\|\p_x^{\tilde{\a}}u\|_{H^1_x}\|\p_x^{\a-\tilde{\a}}h_0+\v\p_x^{\a-\tilde{\a}}\rho\|_{H^1_x}\|\p_x^\a\rho\|_{L^2_x}\\[4pt]
\lesssim&\ \D_{mi,K,1}(t)+\E^{\frac{1}{2}}(t)\D(t),
\end{aligned}
\end{equation*}
and for $B_{322}$, by considering \eqref{estimate of rho0-rho-k}, we have, for $|\beta|=1$,
\begin{equation*}
\begin{aligned}
B_{322}=&\ \frac{2}{2\mu+\lambda}\l\p_x^\a[u(1+h_0+\v\rho)],\p_x^{\a-\beta}u (\p_x h_0+\v\p_x\rho)(1+h_0+\v\rho)\r_x\\[4pt]
\lesssim &\ \|\p_x^\a[u(1+h_0+\v\rho)]\|_{L^2_x}\|\p_x^{\a-\beta}u (\p_x h_0+\v\p_x\rho)\|_{L^2_x}\\[4pt]
\lesssim &\ \left[\|\p_x^\a u\|_{L^2_x}+\|u\|_{L^\infty_x}\|\p_x^\a(1+h_0+\v\rho)\|_{L^2_x}\right] \|\p_x^{\a-\beta}u\|_{L^4_x}\|\p_x h_0 +\v\p_x\rho\|_{L^4_x}\\[4pt]
&\ +\sum_{1 \leq |\tilde{\a}| \leq |\a|-1}\|\p_x^{\tilde{\a}}u\|_{L^4_x}\|\p_x^{\a-\tilde{\a}}(1+h_0+\v\rho)\|_{L^4_x}\|\p_x^{\a-\beta}u\|_{L^4_x}\|\p_x h_0+\v\p_x\rho\|_{L^4_x}\\[4pt]
\lesssim &\ \E^{\frac{1}{2}}(t)\D(t).
\end{aligned}
\end{equation*}
Therefore, we obtain
\begin{equation}\label{B32}
B_{32}\lesssim \D_{mi,K,1}(t)+\E^{\frac{1}{2}}(t)\D(t).
\end{equation}

Recalling the remainder system $\eqref{Remainder equations}_2$, $B_{33}$ can be divided into five parts as follow: for $|\beta| =1 $,
\begin{equation*}\label{B33-0}
\begin{aligned}
B_{33}=&\ \underbrace{-\frac{1}{2\mu+\lambda}\l\p_x^\a\rho,\p_x^{\a-\beta}(u_0\c\n_x u+u\c\n_x u_0)(1+h_0+\v\rho)^2 \r_x}_{B_{331}}\\[4pt]
&\ \underbrace{-\frac{1}{2\mu+\lambda}\l\p_x^\a\rho,\p_x^{\a-\beta} \left(\frac{1}{1+h_0+\v\rho}Lu\right) (1+h_0+\v\rho)^2 \r_x}_{B_{332}}\\
&\ +\underbrace{\frac{1}{(2\mu+\lambda)\v}\l\p_x^\a\rho,\p_x^{\a-\beta} \left[\frac{1}{1+h_0+\v\rho}(b-\v u)\right](1+h_0+\v\rho)^2 \r_x}_{B_{333}}\\[4pt]
&\ +\underbrace{\frac{1}{(2\mu+\lambda)\v}\l\p_x^\a\rho,\p_x^{\a-\beta}\left(\frac{1}{1+h_0+\v\rho}R_2\right)(1+h_0+\v\rho)^2\r_x}_{B_{334}}\\[4pt]
&\ +\underbrace{\frac{1}{2\mu+\lambda}\l\p_x^\a\rho,\p_x^{\a-\beta}\left(\frac{1}{1+h_0+\v\rho}R_3\right)(1+h_0+\v\rho)^2\r_x}_{B_{335}}.
\end{aligned}
\end{equation*}

For $B_{331}$, we have
\begin{equation*}\label{B331}
\begin{aligned}
B_{331}\lesssim &\ \|\p_x^\a \rho\|_{L^2_x}\|\p_x^{\a-\beta}(u_0\c\n_xu + u\c\n_x u_0)\|_{L^2_x}\\[4pt]
\lesssim&\ \left[\|(u_0,u)\|_{H^\a_x}+\|(u_0,u)\|_{H^2_x}\right] \|\p_x^\a \rho\|_{L^2_x} \left[\|(\n_x u_0,\n_x u)\|_{H^{|\a|-1}_x}+\|(\n_x u_0,\n_x u)\|_{H^2_x}\right]\\[4pt]
\lesssim&\ \E^{\frac{1}{2}}(t)\D(t),
\end{aligned}
\end{equation*}
with $|\beta|=1$.

By considering the Lemma \ref{estimate of Presure}, we find, for $B_{332}$,
\begin{equation*}\label{B332}
\begin{aligned}
B_{332}\leq&\ C\sum_{1\leq|\tilde{\a}|\leq|\a|-2}\|\p_x^\a\rho\|_{L^2_x}\|B^{\tilde{\a}}\|_{L^4_x}\|\p_x^{\a-\beta-\tilde{\a}}Lu\|_{L^4_x}-\frac{1}{2\mu+\lambda}\l\p_x^\a\rho,L(\p_x^{\a-\beta}u)(1+h_0+\v\rho)\r_x\\[4pt]
\leq &\ C\|\p_x^\a\rho\|_{L^2_x}(\|\p_x h_0\|_{H^{\a-2}_x}+\v\|\p_x\rho\|_{H^{|\a|-2}_x})\|\p_x u\|_{H^{|\a|-1}_x}+\l\p_x^\a\rho,\p_x^\a\up{div}_xu(1+h_0+\v\rho)\r_x\\[4pt]
&\ +\frac{\mu}{2\mu+\lambda}\l\p_x^{\a-\beta}\rho,(\p_x^\a\up{div}_x u-\Delta_x\p_x^{\a-\beta}u_0) (\p_xh_0+\v\p_x\rho)\r_x\\[4pt]
\leq&\ C\E^{\frac{1}{2}}(t)\D(t)+\l\p_x^\a\rho, \p_x^\a\up{div}_x u(1+h_0+\v\rho)\r_x,
\end{aligned}
\end{equation*}
with $|\beta|=1$.

By recalling the definition of $R_2$ in \eqref{L-R0-R3} and considering \eqref{estimate of rho0-rho-k}, the dissipation of $\p_x^\a\rho$ can be manifested in $B_{335}$, i.e., there exists a constant $C_*>0$ such that
\begin{equation*}\label{B335-0}
\begin{aligned}
B_{335}=&\ -\frac{A}{(2\mu+\lambda)\v}\l\p_x^\a\rho,\Big\{\p_x^\a[(1+h_0+\v\rho)^\g]-\p_x^\a[(1+h_0)^\g] \Big\}(1+h_0+\v\rho)\r_x\\[4pt]
&\ -\frac{A}{(2\mu+\lambda)\v}\sum_{1\leq|\tilde{\a}| \leq |\a|-1}\l\p_x^\a\rho,\p_x^{\tilde{\a}} \left(\frac{1}{1+h_0+\v\rho}\right)  \Big\{\p_x^{\a-\tilde{\a}}[(1+h_0+\v\rho)^\g]\\[4pt]
&\ -\p_x^{\a-\tilde{\a}}[(1+h_0)^\g]\Big\}(1+h_0+\v\rho)^2\r_x\\[4pt]
\leq&\ -2C_*\|\p_x^\a\rho\|^2_{L^2_x}\underbrace{-\frac{A}{(2\mu+\lambda)\v}\l(1+h_0+\v\rho)\p_x^\a\rho,B^\a_{\g}\r_x}_{B_{3351}}\\[4pt]
&\underbrace{-\frac{A}{(2\mu+\lambda)\v}\sum_{1\leq |\tilde{\a}| \leq|\a|-1}C_{\a,\tilde{\a}}\l(1+h_0+\v\rho)^2\p_x^\a\rho,B^{\tilde{\a}}  B^{\a-\tilde{\a}}_\g\r_x}_{B_{3352}}\\[4pt]
&\underbrace{-\frac{A\g}{2\mu+\lambda}\sum_{1\leq|\tilde{\a}|\leq|\a|-1}C_{\a,\tilde{\a}}\l(1+h_0+\v\rho)^{\g+1}\p_x^\a\rho, B^{\tilde{\a}} \p_x^{\a-\tilde{\a}}\rho\r_x}_{B_{3353}},
\end{aligned}
\end{equation*}
where $B^\a_{\g}$ is defined in \eqref{B-la}.

In terms of estimates of $B_{3351}-B_{3353}$, we need to use inequalities \eqref{estimate of 1+h+rho} and \eqref{estimate of nable P-1} in Lemma \ref{estimate of Presure},
\begin{equation*}
\begin{aligned}
\sum_{i=1}^3|B_{335i}| \lesssim\ \E^{\frac{1}{2}}(t)\D(t),
\end{aligned}
\end{equation*}
where \eqref{Key inequality} is also utilized in the last inequality of estimate of $B_{3352}$ and $B_{3353}$ above.\\
Hence, we find
\begin{equation*}\label{B335}
B_{335}\leq \ -2C_*\|\p_x^\a\rho\|^2_{L^2_x} +C\E^{\frac{1}{2}}(t)\D(t).
\end{equation*}

Recalling the definition of $R_3$ in \eqref{L-R0-R3} and applying \eqref{estimate of 1+h+rho} in the Lemma \ref{estimate of Presure}, we obtain the following estimates of $B_{333}$ and $B_{334}$: for $|\beta| = 1$,
\begin{equation*}\label{B333-B334}
\begin{aligned}
B_{333}=&\ \frac{1}{(2\mu+\lambda)\v}\l\p_x^\a\rho,\p_x^{\a-\beta}\left[\frac{1}{1+h_0+\v\rho}(b-\v u)\right](1+h_0+\v\rho)^2 \r_x\\[4pt]
\leq&\ C_*\|\p_x^\a\rho\|^2_{L^2_x}+\frac{C}{\v^2}\|\p_x^{\a-\beta}(b-\v u)\|^2_{L^2_x}\\[4pt]
&\ \qquad \qquad +\frac{C}{\v}\sum_{1\leq|\tilde{\a}|\leq |\a|}\|\p_x^\a\rho\|_{L^2_x}\|B^{\tilde{\a}-\beta}\|_{L^4_x}\|\p_x^{\a-\beta-\tilde{\a}}(b-\v u)\|_{L^4_x}\\[4pt]
\leq&\ C_*\|\p_x^\a\rho\|^2_{L^2_x} + C\left(\D_{mi,K,1}(t)+\E^{\frac{1}{2}}(t)\D(t)\right),\\[8pt]
B_{334}=&\ -\l\p_x^\a\rho,\p_x^\a \left(\frac{R_3}{1+h_0+\v\rho}\right)\r_x\\[4pt]
\lesssim&\ \|\p_x^\a\rho\|_{L^2_x}\|B^\a\|_{L^2_x}\|R_3\|_{L^\infty_x}+\|\p_x^\a\rho\|_{L^2_x}\|B^0\|_{L^\infty_x}\|\p_x^\a R_3\|_{L^2_x}\\[4pt]
&\ \qquad \qquad \qquad \qquad +\sum_{1\leq|\tilde{\a}| \leq|\a|}\|\p_x^\a\rho\|_{L^2_x}\|B^{\tilde{\a}}\|_{L^4_x}\|\p_x^{\a-\tilde{\a}}R_3\|_{L^4_x}\\[4pt]
\lesssim&\ \E^{\frac{1}{2}}(t)\D(t),
\end{aligned}
\end{equation*}
where \eqref{Key inequality} is utilized in the last inequality. %

Combining the estimates of $B_{331}$ - $B_{335}$ above, we obtain
\begin{equation}\label{B33}
B_{33}\leq -C_*\|\p_x^\a\rho\|^2_{L^2_x}+\l\p_x^\a\rho,\p_x^\a\up{div}_x u(1+h_0+\v\rho)\r_x + C\left(\D_{mi,K,1}(t)+\E^{\frac{1}{2}}(t)\D(t)\right).
\end{equation}

Then, by plugging the estimates \eqref{B31}, \eqref{B34}, \eqref{B32}, and \eqref{B33} into \eqref{B3-0}, we obtain that, for $|\beta|=1$,
\begin{multline}\label{B3}
    \frac{1}{2}\frac{d}{dt}\|\p_x^\a\rho\|^2_{L^2_x} + \frac{1}{2\mu + \lambda}\frac{d}{dt}\int \p_x^\a\rho\c\p_x^{\a-\beta}u(1+h_0+\v\rho)^2 dx+C_*\|\p_x^\a\rho\|^2_{L^2_x}\\
\lesssim\big(\D_{mi,K,1}(t)+\E^{\frac{1}{2}}(t)\D(t)\big).
\end{multline}

The proof can be finally completed by summing up the inequality above with $ 1 \leq |\a| \leq 4$.

\end{proof}

In the following lemma, we deduce \textit{a priori} estimate of the mixed partial derivative of $g$, or more precisely, $(\I-\P)g$ (since $\| g\|_{H^4_{x,v}}\sim \|(\I-\P)g\|_{H^4_{x,v}}+\| g\|_{L^2_v H^4_x}$, hence we only need to estimate $\|(\I-\P)g\|_{H^4_{x,v}}$). To achieve this, we first apply the projection operator $\I-\P$ to both-hand-sides of $\eqref{Remainder equations}_1$ in the remainder system, which leads to
\begin{equation}\label{Remainder-equation-I-P-g}
\begin{aligned}
\p_t(\I-\P)g+\frac{1}{\v^2}\L(\I-\P)g=\frac{1}{\v}(\I-\P)R_0+(\I-\P)R_1,
\end{aligned}
\end{equation}
where the fact $\L g=\L(\I-\P)g+\P_1 g$ in \eqref{Decop-L} is used, and $R_0,\,R_1$ are defined in \eqref{L-R0-R3}.

\begin{lemma}\label{estiamte-Mi-kinetic-part-mixed-partial}
Under the assumptions of Theorem \ref{Main-Limits} and \eqref{A priori assumption}, then, for any $0\leq t\leq T$,
\begin{equation}\label{Step four}
\frac{1}{2}\frac{d}{dt}\E_{mi,K,4}(t)+C_4\D_{mi,K,4}(t) \leq \tilde{C}_4\left( \D_{mi,K,1}(t)+\D_{mi,F}(t)+\D_{ma}(t)+\E^{\frac{1}{2}}(t)\D(t)\right),
\end{equation}
where the constants $\tilde{C}_4,\,C_4>0$ are independent of $\v$, $\delta$, $\delta_1$ and $T$.
\end{lemma}

\begin{proof}

Applying $\p_x^\a\p_v^\b$ with $1\leq|\a|+|\b|\leq 4,\,|\b| \geq 1$ to \eqref{Remainder-equation-I-P-g}, multiplying it by $\p_x^\a\p_v^\b(\I-\P) g$, and integrating over $(x,v) \in \mathbb{R}^3 \times \mathbb{R}^3$, we obtain, for $|\beta_1|=1$ and $|\beta_2|=2$,
\begin{equation}\label{Estimate of mix-partial-I-P-g-1}
\begin{aligned}
&\frac{1}{2}\frac{d}{dt}\|\p_x^\a\p_v^\b(\I-\P)g\|^2_{L^2_{x,v}}+\frac{C_0}{\v^2}\|\p_x^\a\p_v^\b(\I-\P)g\|^2_{\nu}\\[4pt]
\leq& \underbrace{-\frac{C_{\b,1}}{2\v^2}\l v\p_x^\a\p_v^{\b-\beta_1}(\I-\P)g,\p_x^\a\p_v^\b(\I-\P)g\r_{x,v}}_{B_{41}} \underbrace{-\frac{C_{\b,2}}{2\v^2}\l \p_x^\a\p_v^{\b-\beta_2}(\I-\P)g,\p_x^\a\p_v^\b(\I-\P)g\r_{x,v}}_{B_{42}}\\
&+\underbrace{\frac{1}{\v}\l\p_x^\a\p_v^\b(\I-\P)R_0,\p_x^\a\p_v^\b(\I-\P)g\r_{x,v}}_{B_{43}} + \underbrace{\l\p_x^\a\p_v^\b(\I-\P)R_1,\p_x^\a\p_v^\b(\I-\P)g\r_{x,v}}_{B_{44}},
\end{aligned}
\end{equation}
where a direct calculation
$$\p_x^\a\p_v^\b\L(\I-\P)g = \L[\p_x^\a\p_v^\b(\I-\P)g]+\frac{C_{\b,1}}{2}v\p_x^\a\p_v^{\b-\beta_1}(\I-\P)g+\frac{C_{\b,2}}{2}\p_x^\a\p_v^{\b-\beta_2}(\I-\P)g,$$
and \eqref{Poin inequality} are used.

For $B_{41}$ and $B_{42}$,
\begin{equation}\label{B41-B42}
\begin{aligned}
|B_{41}| + |B_{42}| = & -\frac{C_{\b,1}}{2 \v^2}\l v\p_x^\a\p_v^{\b-\beta_1}(\I-\P)g, \, \p_x^\a\p_v^\b(\I-\P)g\r_{x,v}\\
&\qquad \qquad \qquad +\frac{C_{\b,2}}{2\v^2}\l\p_x^\a\p_v^{\b-\beta_2}(\I-\P)g, \, \p_x^\a\p_v^\b(\I-\P)g\r_{x,v}\\[4pt]
\leq&\ \frac{c_0}{2^6\v^2}\|\p_x^\a\p_v^\b(\I-\P)g\|_{\nu}^2+\frac{C}{\v^2}\|\p_x^\a\p_v^{\b-\beta_1}(\I-\P)g\|^2_{L^2_{x,v}},
\end{aligned}
\end{equation}
where $|\beta_1|=1$ and $|\beta_2|=2$.

Recalling $R_0$ in \eqref{L-R0-R3} and $(\I-\P_0)(\I-\P)g=(\I-\P)g$ in \eqref{I-P}, $B_{43}$ is divided into the following five parts:
\begin{equation}\label{B43-0}
\begin{aligned}
B_{43}=&\underbrace{-\frac{1}{\v}\l\p_x^\a\p_v^\b(\I-\P)(v\c\n_x g_1), \, \p_x^\a\p_v^\b(\I-\P)g\r_{x,v}}_{B_{431}}\underbrace{-\frac{1}{\v}\l\p_x^\a\p_v^\b(\I-\P)(v\c\n_x g),\, \p_x^\a\p_v^\b(\I-\P)g\r_{x,v}}_{B_{432}}\\[4pt]
&\underbrace{-\frac{1}{\v}\l\p_x^\a\p_v^\b(\I-\P)\left[\frac{1}{\M}\up{div}_v(u_0 g_1 \M)\right],\, \p_x^\a\p_v^\b(\I-\P)g\r_{x,v}}_{B_{433}}\\[4pt]
&\underbrace{-\frac{1}{\v}\l\p_x^\a\p_v^\b(\I-\P)\left[\frac{1}{\M}\up{div}_v(u_0 g \M)\right], \, \p_x^\a\p_v^\b(\I-\P)g\r_{x,v}}_{B_{434}}\\[4pt]
&+\underbrace{\frac{1}{\v}\l\p_x^\a\p_v^\b(\I-\P)\left[\frac{1}{\M}\up{div}_v(u m_0 M)\right], \, \p_x^\a\p_v^\b(\I-\P)g\r_{x,v}}_{B_{435}}.
\end{aligned}
\end{equation}
For $B_{431}$, by plugging in the definition of $g_1$,
\begin{equation}\label{B431}
\begin{aligned}
|B_{431}| \leq&\ \frac{1}{\v}\|\p_x^\a\p_v^\b(v\c\n_x g_1)\|_{L^2_{x,v}}\|\p_x^\a\p_v^\b(\I-\P)g\|_{L^2_{x,v}}\\[4pt]
\leq&\ \frac{c_0}{2^6\v^2}\|\p_x^\a\p_v^\b(\I-\P)g\|^2_{L^2_{x,v}} + C\left(\D_{ma}(t)+\E^{\frac{1}{2}}(t)\D(t)\right).
\end{aligned}
\end{equation}
For $B_{432}$, we have
\begin{equation*}
\begin{aligned}
|B_{432}| = &\ \left| -\frac{1}{\v}\l\p_x^\a\p_v^\b(\I-\P)(v\c\n_x g), \, \p_x^\a\p_v^\b(\I-\P)g\r_{x,v} \right| \\[4pt]
=&\ \Big|  \underbrace{-\frac{1}{\v}\l\p_x^\a\p_v^\b [v\c\n_x (\I-\P)g],\p_x^\a\p_v^\b(\I-\P)g\r_{x,v}}_{B_{4321}} \underbrace{-\frac{1}{\v}\l\p_x^\a\p_v^\b [v\c\n_x \P g],\p_x^\a\p_v^\b(\I-\P)g\r_{x,v}}_{B_{4322}}\\[4pt]
&\ +\underbrace{\frac{1}{\v}\l\p_x^\a\p_v^\b \P (v\c\n_x g),\p_x^\a\p_v^\b(\I-\P)g\r_{x,v}}_{B_{4323}} \Big|,
\end{aligned}
\end{equation*}
where $B_{4321}$ and $B_{4322}$ can be estimated as follows: for $|\beta_1|=1$,
\begin{equation*}
\begin{aligned}
|B_{4321}| = &\ \left| -\frac{C_{\b,1}}{\v}\l\p_x^{\a+\beta_1}\p_v^{\b-\beta_1}(\I-\P)g,\, \p_x^\a\p_v^\b(\I-\P)g\r_{x,v} \right| \\[4pt]
\leq& \ \frac{c_0}{2^8\v^2}\|\p_x^\a\p_v^\b(\I-\P)g\|^2_{L^2_{x,v}}+C\|\p_x^{\a+\beta_1}\p_v^{\b-\beta_1}(\I-\P)g\|^2_{L^2_{x,v}},\\[8pt]
|B_{4322}| =&\ \left| -\frac{1}{\v}\l\p_x^\a\p_v^\b[(v\c\n_x a+v\otimes v:\n_x b)\M],\, \p_x^\a\p_v^\b(\I-\P)g\r_{x,v} \right|\\[4pt]
\leq&\  \frac{c_0}{2^8\v^2}\|\p_x^\a\p_v^\b(\I-\P)g\|^2_{L^2_{x,v}} + C\|(\p_x^{\a+\beta_1}a,\p_x^{\b+\beta_1}b)\|^2_{L^2_{x,v}}\\[4pt]
\leq&\ \frac{c_0}{2^8\v^2}\|\p_x^\a\p_v^\b(\I-\P)g\|^2_{L^2_{x,v}}+C\D_{mi,F}(t),
\end{aligned}
\end{equation*}
furthermore, noting that
\begin{equation*}
\begin{aligned}
\P(v\c\n_x g)=&\ \P[v\c\n_x\P g]+\P[v\c\n_x(\I-\P)g]\\[4pt]
=&\ (v\c\n_x a+\up{div}_x b)\M+\P[v\c\n_x(\I-\P)g]\\[4pt]
=&\ \left[v\c\n_x a + \up{div}_x b + \int_{\R^3} v\c\n_x(\I-\P)g \M dv + v \int_{\R^3} v\c\n_x(\I-\P)g  v\M dv\right]\M,
\end{aligned}
\end{equation*}
we find, for $B_{4323}$,
\begin{equation*}
\begin{aligned}
|B_{4323}| =&\ \left| \frac{1}{\v}\l\p_x^\a\p_v^\b \P v\c\p_x g, \, \p_x^\a\n_v^\b(\I-\P)g\r_{x,v} \right| \\[4pt]
\leq &\ \frac{c_0}{2^8\v^2}\|\p_x^\a\p_v^\b(\I-\P)g\|^2_{L^2_{x,v}}+C\|(\p_x^{\a+\beta_1} a,\p_x^\a\up{div}_x b)\|^2_{L^2_x}+C\|\p_x^{\a+\beta_1}(\I-\P)g\|^2_{L^2_{x,v}}\\[4pt]
\leq &\ \frac{c_0}{2^8\v^2}\|\p_x^\a\p_v^\b(\I-\P)g\|^2_{L^2_{x,v}} + C\left(\D_{ma,F}(t)+\D_{mi,K,1}(t)\right),
\end{aligned}
\end{equation*}
therefore, we obtain, for $|\beta_1| = 1$,
\begin{equation}\label{B432}
|B_{432}| \leq \ \frac{c_0}{2^6\v^2}\|\p_x^\a\p_v^\b(\I-\P)g\|^2_{L^2_{x,v}} + C\left[\|\p_x^{\a+\beta_1}\p_v^{\b-\beta_1}(\I-\P)g\|^2_{L^2_{x,v}}+\D_{ma,F}(t)+\D_{mi,K,1}(t)\right].
\end{equation}
By considering the fact that
$$\P\left[\frac{1}{\M}\up{div}_v(u_0 g_1 \M)\right]=0, \quad \P \left[\frac{1}{\M}\up{div}_v(u_0 g \M)\right] = -v\c u_0 a\M,$$
we can obtain the following estimates for $B_{433}$, $B_{343}$ and $B_{435}$, respectively,
\begin{equation}
\begin{aligned}
|B_{433}|+|B_{434}|\lesssim\E^{\frac{1}{2}}(t)\D_{ma}(t).\\[8pt]
\end{aligned}
\end{equation}

\begin{equation}
\begin{aligned}
|B_{435}| =&\ \Big|  -\frac{1}{\v}\l\p_x^\a\p_v^\b\left[\frac{1}{\M}\up{div}_v(u m_0 M)\right], \, \p_x^\a\p_v^\b(\I-\P)g\r_{x,v}\\[4pt]
& \ \qquad \qquad \qquad +\frac{1}{\v}\l\p_x^\a\p_v^\b(v\c u m_0\M),\, \p_x^\a\p_v^\b(\I-\P)g\r_{x,v} \Big| \\[4pt]
\leq&\ \frac{1}{\v}\|\p_x^\a\p_v^\b \left[\frac{1}{\M}\up{div}_v(u m_0 M)\right]\|_{L^2_{x,v}}\|\p_x^\a\p_v^\b(\I-\P)g\|_{L^2_{x,v}} \\[4pt]
&\ \qquad \qquad \qquad +\frac{1}{\v}\|\p_x^\a\p_v^\b(v\c u m_0\M)\|_{L^2_{x,v}}\|\p_x^\a\p_v^\b(\I-\P)g\|_{L^2_{x,v}}\\[4pt]
\lesssim&\ \E^\frac{1}{2}(t)\D_{ma}(t).
\end{aligned}
\end{equation}
By substituting the estimates of $B_{431}$ to $B_{435}$ above into \eqref{B43-0}, we can obtain, for $|\beta_1| = 1$,
\begin{multline}\label{B43}
|B_{43}| \leq \  C \left[\|\p_x^{\a+\beta_1}\p_v^{\b-\beta_1}(\I-\P)g\|^2_{L^2_{x,v}}
+\D_{ma,F}(t)+\D_{mi,K,1}(t)+\E^\frac{1}{2}(t)\D_{ma}(t)\right] \\
+ \frac{c_0}{2^5\v^2}\|\p_x^\a\p_v^\b(\I-\P)g\|^2_{L^2_{x,v}}.
\end{multline}

Next, recalling the definition of $R_1$ in \eqref{L-R0-R3}, we can follow a similar estimate of $B_{43}$ and obtain that, for $B_{44}$,
\begin{equation}\label{B44}
|B_{44}| \leq\ \frac{c_0}{2^6\v^2}\|\p_x^\a\p_v^\b(\I-\P)g\|^2_{L^2_{x,v}} + C\left(\D_{ma,F}(t)+\E^\frac{1}{2}(t)\D_{ma}(t)\right).
\end{equation}

Though note that the quantity $\frac{C}{\v^2}\|\n_x^\a\n_v^{\b-1}(\I-\P)g\|^2_{L^2_{x,v}}$ in \eqref{B41-B42} is still not under control, however, by observing that the order of $v$-derivatives in this quantity is $|\b|-1$, we can take summation over $|\b|$ and substitute the estimates \eqref{B41-B42}, \eqref{B43}, \eqref{B44} into \eqref{Estimate of mix-partial-I-P-g-1}, we finally obtain

\begin{equation}\label{Estimate of mix-partial-I-P-g}
\frac{1}{2}\frac{d}{dt}\E_{mi,K,4}(t)+C_{1,4}\D_{mi,K,4}(t)\lesssim\D_{mi,K,1}(t)+\D_{mi,F}(t)+\D_{ma}(t)+\E^{\frac{1}{2}}(t)\D(t).
\end{equation}

\end{proof}

\subsubsection{Estimates of fluid part of the remainder system}
\label{subsubsec:fluid_remainder}

The goal of this subsection is to obtain the energy estimate of the fluid part of the remainder system \eqref{Remainder equations}.
To achieve this, we propose the following hyperbolic-parabolic coupled system: for $1\leq i,j\leq 3$,
\begin{equation}\label{Pg-a-b}
\left\{
\begin{aligned}
&\v\p_t a+\up{div}_x b=0,\\[4pt]
&\p_t b_i+\frac{1}{\v}\p_{x_i} a+ \frac{1}{\v^2} (b_i-\v u_i) + \frac{1}{\v}\sum_{j=1}^3\Gamma_{ij}[\p_{x_j}(\I-\P)g]=\frac{1}{\v}R_{4i}+R_{5i},\\[4pt]
\end{aligned}
\right.
\end{equation}
where $\Gamma_{ij}(g)$, $R_{4i}$, and $R_{5i}$ are defined as follows:
\begin{equation}\label{R4i-R5i-Gamma-l}
\begin{aligned}
R_{4i}=& \ u_{0i}a+u_i m_0,\\[4pt]
R_{5i}=&\ -\big[\p_t(u_{0i} m_0)+u_i a+\p_t u_{0i}+\p_t\p_{x_i} m_0\big],\\[4pt]
\Gamma_{ij}(g)=&\ \int g(v_i v_j-1)\M dv.
\end{aligned}
\end{equation}

Note that, in contrast with the hyperbolic-parabolic coupled system in \cite{CJADRJMA11}, there is no gradient term of $b$ in the new system above, 
since $\frac{1}{\v^2}\|\n_x b\|_{H^3_x}\leq \frac{1}{\v^2}\|\n_x b-\v\n_x u\|_{H^3_x}+\|\n_x u\|_{H^3_x}$, and the terms $\frac{1}{\v^2}\|\n_x b-\v\n_x u\|_{H^3_x}$, $\|\n_x u\|_{H^3_x}$ have been included in the dissipative parts shown in Lemma \ref{estimate-kinetic-Mi-1}.

Now, we are in a position to present the following lemma about the energy estimate of the fluid part in the remainder system \eqref{Remainder equations}.

\begin{lemma}\label{estiamte-Mi-fluid-part}
Under the assumptions of Theorem \ref{Main-Limits} and \eqref{A priori assumption}, then, for any $0\leq t\leq T$,
\begin{equation}\label{Last step}
\begin{aligned}
\frac{1}{2}\frac{d}{dt}\E_{mi,F}(t)+C_5\D_{mi,F}(t)\leq \tilde{C}_5 \left( \D_{mi,K,1}(t)+\D_{ma}(t)+\E^{\frac{1}{2}}(t)\D(t) \right),
\end{aligned}
\end{equation}
where the constants $\tilde{C}_5,\,C_5 > 0$ are independent of $\v$, $\delta$, $\delta_1$ and $T$.
\end{lemma}

\begin{proof}
By applying the derivative operator $\p_x^\a$ with $0\leq|\a|\leq 3$ to both-hand-sides of the hyperbolic-parabolic system $\eqref{Pg-a-b}_2$, then multiplying it by $\p_x^\a\p_{x_i} a$, and integrating over $x \in \R^3$, we obtain, %

\begin{equation}\label{B7-0}
\begin{aligned}
\|\p_x^{\a+\beta}a\|^2_{L^2_x}=& \sum_{i=1}^3\l\p_x^\a\p_{x_i} a, \, \p_x^\a\p_{x_i} a\r_x\\[4pt]
=& \ -\v\frac{d}{dt} \int_{\R^3} \p_x^{\a+\beta} a \, \p_x^\a b \,dx + \underbrace{\v\sum_{i=1}^3\l\p_x^\a\p_{x_i}\p_t a, \, \p_x^\a b_i\r_x}_{B_{71}}\\[4pt]
&\ \underbrace{-\frac{1}{\v}\sum_{i=1}^3\l\p_x^\a\p_{x_i} a, \, \p_x^\a(b_i-\v u_i)\r_x}_{B_{72}} \underbrace{-\sum_{i,j=1}^3\l\p_x^\a\p_{x_i} a,\, \p_x^\a\p_{x_i}(\I-\P)g(v_iv_j-1)\M\r_{x,v}}_{B_{73}}\\[4pt]
&\ +\underbrace{\sum_{i=1}^3\l\p_x^\a\p_{x_i} a, \, \p_x^\a R_{4i}\r_x}_{B_{74}} + \underbrace{\sum_{i=1}^3\l\p_x^\a\p_{x_i} a, \, \p_x^\a R_{5i}\r_x}_{B_{75}},
\end{aligned}
\end{equation}
for $0\leq|\a|\leq 3$ and $|\beta| = 1$.\\[4pt]
By substituting $\p_t a$ by $\eqref{Pg-a-b}_1$, we have, for $B_{71}$,
\begin{equation}\label{B71}
|B_{71}| = \left| -\sum_{i=1}^3\l\p_x^\a\p_{x_i}\up{div}_x b,\p_x^\a b_i\r_x \right| = \|\p_x^\a\up{div}_x b\|^2_{L^2_x},
\end{equation}
and $B_{72}$, by using Young's inequality, we have
\begin{equation}\label{B72}
\begin{aligned}
|B_{72}| =&\ \left| -\frac{1}{\v}\sum_{i=1}^3\l\p_x^\a\p_{x_i} a,\p_x^\a(b_i-\v u_i)\r_x \right|  \\[4pt]
\leq & \ \frac{1}{16}\|\p_x^{\a+\beta} a\|^2_{L^2_x}+\frac{C}{\v^2}\|\p_x^\a(b-\v u)\|^2_{L^2_x}\\[4pt]
\leq&\ \frac{1}{16}\|\p_x^{\a+\beta} a\|^2_{L^2_x}+C\D_{mi,K,1}(t),
\end{aligned}
\end{equation}
and, for $B_{73}$,
\begin{equation}\label{B73}
\begin{aligned}
|B_{73}| \leq&\ \frac{1}{16}\|\p_x^{\a+\beta} a\|^2_{L^2_x} + \frac{C}{\v^2}\|\p_x^{\a+\beta}(\I-\P)g\|^2_{L^2_{x,v}}\\[4pt]
\leq&\ \frac{1}{16}\|\p_x^{\a+\beta} a\|^2_{L^2_x}+C\D_{mi,K,1}(t),
\end{aligned}
\end{equation}
Furthermore, recalling the definition of $R_{4i}$ and $R_{5i}$ in \eqref{R4i-R5i-Gamma-l} and noting \eqref{A priori assumption}, we have
\begin{equation}\label{B74}
\begin{aligned}
|B_{74}| =&\ \sum_{i=1}^3 \l\p_x^\a\p_{x_i} a, \, \p_x^\a(u_{0i}a+u_i m_0)\r_x\\[4pt]
\lesssim& \ \|\p_x^{\a+\beta}a\|_{L^2_x}\|u_0\|_{H^{|\a|}_x}\|\n_x a\|_{H^{|\a|-1}_x} + C\|\p_x^{\a+\beta}a\|_{L^2_x}\|m_0\|_{H^{|\a|}_x}\|\n_x u\|_{H^{|\a|-1}_x}\\[4pt]
\lesssim&\ \E^{\frac{1}{2}}(t)\D(t)),
\end{aligned}
\end{equation}
and
\begin{equation}\label{B75}
\begin{aligned}
|B_{75}| = &\ \left| -\sum_{i=1}^3\l\p_x^\a\p_{x_i} a, \, \p_x^\a \left[\p_t(u_{0i} m_0)+u_i a+\p_t u_{0i}+\p_t\p_{x_i} m_0\right]\r_x \right| \\[4pt]
\lesssim&\ \|\p_x^{\a+\beta}a\|_{L^2_x} \left(\|\p_t u_0\|_{H^{|\a|}_x}\|\n_x m_0\|_{H^{|\a|-1}_x} + \|\p_t m_0 \|_{H^{|\a|}_x}\|\p_x u_0\|_{H^{|\a|-1}_x} \right)\\[4pt]
&\ +\|\p_x^{\a+\beta}a\|_{L^2_x}\|\p_t\p_x^\a u_0\|_{L^2_x}\|\p_t\p_x^{\a+1} m_0\|_{L^2_x}\\[4pt]
\leq&\ \frac{1}{16}\|\p_x^{\a+\beta} a\|^2_{L^2_x}+C(\|\p_t\p_x^\a u_0\|^2_{L^2_x}+\|\p_t\p_x^{\a+\beta} m_0\|^2_{L^2_x})+C\E^{\frac{1}{2}}(t)\D(t)\\[4pt]
\leq&\ \frac{1}{16}\|\p_x^{\a+\beta} a\|^2_{L^2_x} + C\left(\D_{ma}(t)+\E^{\frac{1}{2}}(t)\D(t)\right).
\end{aligned}
\end{equation}\\
Hence, by inserting the estimate of $B_{71}$ - $B_{75}$ into \eqref{B7-0}, we find
\begin{equation}\label{B7}
\begin{aligned}
\frac{7}{8}\|\p_x^{\a+\beta}a\|^2_{L^2_x}+\v\frac{d}{dt}\int_{\R^3} \p_x^{\a+\beta} a \, \p_x^\a b \,dx \ \lesssim \ \D_{ma}(t) + \D_{mi,K,1}(t) +\E^{\frac{1}{2}}(t)\D(t).
\end{aligned}
\end{equation}

The proof can be finally completed by summing up with $|\a|=0,1,2,3$ in \eqref{B7}.

\end{proof}

\subsection{Proof of total energy estimate}
\label{subsec:total_energy}

In this subsection, we present the detailed proof of the total energy estimate, i.e., Proposition \ref{A priori estimates}, by collecting the previous energy estimate Lemmas \ref{estimate-kinetic-Mi-1}, \ref{estimate-kinetic-Mi-2}, \ref{estimate-kinetic-Mi-3}, \ref{estiamte-Mi-kinetic-part-mixed-partial} and \ref{estiamte-Mi-fluid-part} altogether.

Noting the definitions of $\E_{mi,F}(t),\,\E_{mi,K,3}(t)$ and $\E_{mi,K,1}(t)$ in \eqref{part of energy functionals}, we can find that there exists a constant $\tilde{C}_6>0$ such that
\begin{equation}\label{C1-6}
|\E_{m,F}(t)+\E_{mi,K,3}(t)|\leq \tilde{C}_6\E_{mi,K,1}(t).
\end{equation}

Next, we prove the total energy estimate \eqref{total_energy_prop} by the following four steps:

\textbf{Step 1:} To eliminate $\D_{mi,K,3}(t)$ in \eqref{Step two} of the Lemma \ref{estimate-kinetic-Mi-2}, we choose a constant $\tilde{\lambda}_1>0$ large enough such that
\begin{equation}\label{tilde-lambda-1}
\frac{\tilde{C}_3}{2} \tilde{\lambda}_1 \geq \  \tilde{C}_2,
\end{equation}
where $\tilde{C}_3$ is the constant in Lemma \ref{estimate-kinetic-Mi-3}.

By taking $\eqref{Step two} + \tilde{\lambda}_1 \times  \eqref{Step three} + \eqref{Last step}$, we find that there exists constants $C_7,\tilde{C}_7 > 0 $ such that
\begin{multline}\label{First time}
\frac{1}{2}\frac{d}{dt} \left( \tilde{\lambda}_1\E_{mi,K,2}(t) +\E_{mi,K,3}(t)+\E_{mi,F}(t)\right) + C_7 \left(\sum_{i=2}^3\D_{mi,K,i}(t)+\D_{mi,F}(t)\right)\\[4pt]
 \leq \tilde{C}_7 \left( \D_{mi,K,1}(t)+\D_{ma}(t) + \E^{\frac{1}{2}}(t)\D(t)\right).
\end{multline}

\textbf{Step 2:} To eliminate $\D_{mi,K,1}(t)$ in inequality \eqref{First time}, we choose a constant $\tilde{\lambda}_2>0$ large enough such that
\begin{equation}\label{tilde-lambda-2}
\frac{C_1\tilde{\lambda}_2}{2} \geq  \left( \tilde{C}_7 + \tilde{\lambda}_1\right) \left( \tilde{C}_6+1 \right),
\end{equation}
where $C_1$, $\tilde{C}_7$,\,$\tilde{C}_6$ are constants in Lemma \ref{estimate-kinetic-Mi-1}, \eqref{First time}, and \eqref{C1-6} respectively.

Then, by applying $ \tilde{\lambda}_2 \times \eqref{Step one} + \eqref{First time}$, there exist $C_8,\tilde{C}_8>0$ such that
\begin{multline}\label{Second time}
\frac{1}{2}\frac{d}{dt} \left(\tilde{\lambda}_2 \E_{mi,K,1}(t) + \tilde{\lambda}_1\E_{mi,K,2}(t) + \E_{mi,K,3}(t)+\E_{mi,F}(t) \right)\\[4pt]
+ C_8\left(\sum_{i=1}^3\D_{mi,K,i}(t)+\D_{mi,F}(t)\right) \leq \tilde{C}_8 \left(\D_{ma}(t)+\E^{\frac{1}{2}}(t)\D(t) \right).
\end{multline}

\textbf{Step 3:} To eliminate $\D_{ma}(t)$ in inequality \eqref{Second time}, we choose choose a constant $\tilde{\lambda}_3>0$ large enough such that
\begin{equation}\label{tilde-lambda-3}
\frac{C_0}{2}\tilde{\lambda}_3\geq \tilde{C}_8,
\end{equation}
where $C_0$ is the constant in Proposition \ref{Solu-NSS}.

Then, by using $\eqref{Second time} + \tilde{\lambda}_3 \times \eqref{Macro-part}$, it turns out that one can find $C_9,\tilde{C}_9>0$ such that
\begin{multline}\label{Third time}
\frac{1}{2}\frac{d}{dt}\big(\tilde{\lambda}_2\E_{mi,K,1}(t)+\tilde{\lambda}_1\E_{mi,K,2}(t)+\E_{mi,K,3}(t)+\E_{mi,F}(t)+\tilde{\lambda}_3\E_{ma}(t)\big)\\[4pt]
+C_9\big(\sum_{i=1}^3\D_{mi,K,i}(t)+\D_{mi,F}(t)+\D_{ma}(t)\big)\leq \tilde{C}_9\E^{\frac{1}{2}}(t)\D(t).
\end{multline}

\textbf{Step 4:} To eliminate $\D_{mi,K,1}(t),\,\D_{mi,F}(t),\,\D_{ma}(t)$ in inequality \eqref{Step four} of Lemma \ref{estiamte-Mi-kinetic-part-mixed-partial}, we choose a constant $\tilde{\lambda}_4>0$ large enough such that
\begin{equation}\label{tilde-lambda-4}
\frac{C_8\tilde{\lambda}_4}{2}\geq \tilde{C}_4,
\end{equation}
where $\tilde{C}_4$ is the constant in Lemma \ref{estiamte-Mi-kinetic-part-mixed-partial}.

Therefore, by applying $\tilde{\lambda}_4 \times \eqref{Third time} + \eqref{Step four}$, we can show that there exist $C_{10},\,\tilde{C}_{10}>0$ such that
\begin{multline}\label{Last time}
\frac{1}{2}\frac{d}{dt} \left(\tilde{\lambda}_4\tilde{\lambda}_2\E_{mi,K,1}(t)+\tilde{\lambda}_4\tilde{\lambda}_1\E_{mi,K,2}(t)+\tilde{\lambda}_4\E_{mi,K,3}(t) + \E_{mi,K,4}(t)+\tilde{\lambda}_4\E_{mi,F}(t)+\tilde{\lambda}_4\tilde{\lambda}_3\E_{ma}(t)\right)\\[4pt]
+ C_{10} \left(\sum_{i=1}^4\D_{mi,K,i}(t)+\D_{mi,F}(t)+\D_{ma}(t)\right) \leq \tilde{C}_{10} \E^{\frac{1}{2}}(t)\D(t).
\end{multline}

Thus, by defining
\begin{equation*}
\begin{aligned}
\E(t)&: =\tilde{\lambda}_4\tilde{\lambda}_2\E_{mi,K,1}(t)+\tilde{\lambda}_4\tilde{\lambda}_1\E_{mi,K,2}(t)+\tilde{\lambda}_4\E_{mi,K,3}(t)+\E_{mi,K,4}(t)+\tilde{\lambda}_4\E_{mi,F}(t)+\tilde{\lambda}_4\tilde{\lambda}_3\E_{ma}(t),\\[4pt]
\D(t)&: = \sum_{i=1}^4\D_{mi,K,i}(t)+\D_{mi,F}(t)+\D_{ma}(t),
\end{aligned}
\end{equation*}
and re-naming $\lambda_i$, $i=1,\cdots,6$ as follows:
\begin{equation}\label{lambda constants}
\begin{aligned}
\lambda_1:=\tilde{\lambda}_4\tilde{\lambda}_2,\quad \lambda_2:=\tilde{\lambda}_4\tilde{\lambda}_1,\quad
\lambda_3:=\tilde{\lambda}_4, \quad
\lambda_4:=1, \quad
\lambda_5:=\tilde{\lambda}_4,\quad
\lambda_6:=\tilde{\lambda}_4\tilde{\lambda}_3,
\end{aligned}
\end{equation}
we obtain
\begin{equation}
    \frac{1}{2} \frac{d \E(t)}{d t} + C_{10} \D(t) \leq \tilde{C}_{10}\E^{\frac{1}{2}}(t)\D(t).
\end{equation}
Thus, the total energy estimate \eqref{total_energy_prop} can be finally obtained by integration in time as well as considering the Proposition \ref{Solu-NSS}, \eqref{equivalent} and \eqref{A priori assumption}.

\section{Proof of main results}
\label{sec:proof_of_main}

\subsection{Proof of Theorem \ref{Main-Limits}}
\label{subsec:proof_main_Theorem}

In this subsection, we complete the proof of Theorem \ref{Main-Limits} by extending the local well-posedness to global-in-time.

Let $T^*_\v$ denote the maximal existence time for the solution to the Cauchy problem (\ref{Remainder equations}) for any given $0<\v\leq \v_0$ with $\v_0$ given in Proposition \ref{Local-in-time solution of the remainder equations} with $M_0$ replaced by $\delta_1$. Proposition \ref{Local-in-time solution of the remainder equations} implies that $T^*_\v>0$. If $T^*_\v=\infty$ for any given $\v$, the global existence of the solutions is obtained. For otherwise, i.e., $T^*_\v\in(0,\infty)$ for some $\v\in (0,\v_0)$, we will prove that it leads to a contradiction by continuity argument.

In fact, using the initial assumption $\mathbb{E}(0)\leq\delta$ in Theorem \ref{Main-Limits} and the continuity of $\mathbb{E}(t)$ in $(0,T_\v^*)$, then there exists a $\bar{T}_\v^*\in (0,T_\v^*)$ such that
\begin{equation}\label{barT}
\sup_{0\leq t\leq\bar{T}_\v^*} \big(\|g\|_{H^4_{x,v}}^2+\|(u,\rho)\|_{H^4_x}^2\big) \leq \delta_1,
\end{equation}where $\delta_1>\max\{\delta,2C^*\delta\}$ for some constants $\delta,\,C^*$ given in \eqref{Initial data} and Corollary \ref{Instant energy-dissipative rate}.

Then using Corollary \ref{Instant energy-dissipative rate}, we have
\begin{equation}\label{Iteration results}
\sup_{0\leq t\leq \bar{T}_\v^*}\mathbb{E}(t)=\sup_{0\leq t\leq \bar{T}_\v^*}\big(\|g\|_{H^4_{x,v}}^2+\|(u,\rho)\|_{H^4_x}^2+\|m_0\|_{H^6_x}^2+\|(u_0, h_0)\|_{H^5_x}^2\big)\leq C^*\delta<\frac{1}{2}\delta_1.
\end{equation}
By using the standard continuity argument, (\ref{barT}) and (\ref{Iteration results}) yield that
\begin{equation}\label{barT+1}
\|g(t)\|_{H^4_{x,v}}^2+\|(u,\rho)(t)\|_{H^4_x}^2 \leq \delta_1,
\end{equation}
 for all $t\in[0,T^*_\v)$, provided that $\mathbb{E}(0)\leq\delta$ where $\max\{\delta,2C^*\delta\}<\delta_1$.

Choosing $T^*_0=T^*_\v-\frac{T_\v(\delta_1)}{2}$ as initial time where $T_\v(\delta_1)$ is determined by Proposition \ref{Local-in-time solution of the remainder equations} and (\ref{barT+1}), we obtain from Proposition \ref{Local-in-time solution of the remainder equations} that the solution can be extended to $[0,T_\v^*+\frac{T_\v(\delta_1)}{2})$ for any given $0<\v \leq \v_0 := \v_0(\delta_1)$, which is a contradiction with the definition of $T_\v^*$. Thus $T_\v^*=\infty$. The proof of the global existence of the solution is complete. Since the proof of the uniqueness is standard, we omit the detail for brevity.

\subsection{Proof of Corollary \ref{Conergence}}
\label{subsec:proof_main_corollary}

In this subsection, we present the proof of Corollary \ref{Conergence} by using the Sobolev embedding theorems.

By recalling the expansion form of $f^{\v}$
$$f^\v=(1+m_0)M + \v[(v\c u_0)(1+m_0)-v\c\n_x m_0\big]M + \v g\M,$$
and applying the Sobolev embedding $H^2\hookrightarrow L^\infty$, we have
\begin{equation}
\begin{aligned}
|f^\v-(1+m_0)M|\leq&\ \sup_{t\geq 0}\|f^\v-(1+m_0)M\|_{L^\infty_xL^\infty_v}\\[4pt]
\lesssim &\ \sup_{t\geq 0}\|f^\v-(1+m_0)M\|_{H^2_x H^2_v} \\[4pt]
\lesssim &\ \v \sup_{t\geq 0}\|[(v\c u_0)(1+m_0)-v\c\n_x m_0\big]M+g\M\|_{H^2_x H^2_v}\\[4pt]
\lesssim &\ \v\sup_{t\geq 0}\| \left[ (v\c u_0)(1+m_0)-v\c\n_x m_0\right]M\|_{H^2_xH^2_v} +  \v \sup_{t\geq 0} \|g \sqrt{M} \|_{H^2_x L^2_v}\\[4pt]
&\ + \v \sup_{t\geq 0}\|\n_v \left[ (\I-\P)g\M \right]\|_{H^2_x H^1_v} + \v\sup_{t \geq 0} \|\n_v(\P g \sqrt{M})\|_{H^2_x H^1_v}\\[4pt]
\lesssim &\  \v\sup_{t\geq 0} \Big[\|u_0\|_{H^2_x}(1+\|m_0\|_{H^2_x})+\|\n_xm_0\|_{H^2_x} + \|g \|_{H^2_x L^2_v}\\[4pt]
&\ +\|\n_v(\I-\P)g\|_{H^2_xH^1_v}+\|(a,b)\|_{H^2_x}\Big] \\[4pt]
\lesssim & \v,
\end{aligned}
\end{equation}
where the decomposition $g=(\I-\P)g+\P g=(\I-\P)g+(a+v\c b)\M$ is utilized.
Similarly, considering the expansion form $u^\v=u_0+\v u$ and $\rho^\v=(1+h_0)+\v\rho$, we also obtain
\begin{equation}
\begin{split}
 |u^\v-u_0| \leq& \ \v\sup_{t\geq 0}\|u\|_{L^\infty_x}\lesssim \v\sup_{t\geq 0}\|u\|_{H^2_x}\lesssim \v, \\[4pt]
    |\rho^\v- (1+h_0)|  \leq &\  \v\sup_{t\geq 0}\|\rho\|_{L^\infty_x}\lesssim \v\sup_{t\geq 0}\|\rho\|_{H^2_x}\lesssim \v,
\end{split}
\end{equation}
where the uniform boundedness of all quantities, i.e., $(g,u)$, $(m_0,u_0,h_0)$, $(a,b)$, is guaranteed by the global-in-time energy estimate \eqref{The global estimate}.

\appendix

\section{Proof of Lemma \ref{estimate of Presure}}
\label{sec:appdenix-pressure}

In the appendix, we present the complete proof of the important Lemma \ref{estimate of Presure} in dealing with the pressure term.

\begin{proof}
By \eqref{m0u0h0}, and Sobolev embedding relation, we can find that
$$\frac{1}{2}\leq \sup_{0\leq t\leq T}\|1+h_0\|_{L^\infty_x}\leq \frac{3}{2}.$$

Also, noticing the $0<\v\leq 1$ and \eqref{A priori assumption}, then we can obtain
\begin{equation}\label{estimate of rho0-rho-k}
\begin{aligned}
\frac{1}{4}\leq \sup_{0\leq t\leq T}\|1+h_0+\v\rho\|_{L^\infty_x}\leq 2.
\end{aligned}
\end{equation}

By using \eqref{A priori assumption} and $0 < \v\leq 1$, we have
\begin{equation}\label{Key inequality}
\v \left(\|\rho\|_{L^\infty_x}+\|\n_x\rho\|_{H^1_x}+\|\n_x\rho\|_{H^2_x}\right) \leq  \frac{3}{8}.
\end{equation}

We only prove the inequality \eqref{estimate of 1+h+rho} in the case of $1\leq |\a|\leq 4$ and $p=2$.

For the sake of convenience, we denote $\n^k_x$ as any $x-$partial derivative $\p^{\a}_x$ with the multi-index $|\a|=k \in \mathbb{N}$.

Applying \eqref{A priori assumption}, \eqref{estimate of rho0-rho-k} and \eqref{Key inequality}, we have
\begin{equation*}\label{estimate-Lp-h}
\begin{aligned}
\|\n_x \left(\frac{1}{1+h_0+\v\rho}\right) \|_{L^2_x}\lesssim&\ \|\n_x h_0\|_{L^2_x}+\v\|\n_x\rho\|_{L^2_x},\\[4pt]
\|\n_x^2 \left(\frac{1}{1+h_0+\v\rho}\right) \|_{L^2_x}\lesssim& \ \left(\|\n_x^2 h_0\|_{L^2_x}+\v\|\n_x^2\rho\|_{L^2_x}\right) + \left( \|\n_x h_0\|_{L^4_x}+\v\|\n_x\rho\|_{L^4_x} \right)^2\\
\lesssim&\ \sum_{i=1}^2(\|\n_x h_0\|_{H^1_x}^i+\v^i\|\n_x\rho\|_{H^1_x}^i)\\
\lesssim&\ \|\n_x h_0\|_{H^1_x}+\v\|\n_x\rho\|_{H^1_x},\\[4pt]
\|\n_x^3 \left(\frac{1}{1+h+\v\rho}\right) \|_{L^2_x}
\lesssim&\ \left(\|\n_x^3 h_0\|_{L^2_x}+\v\|\n_x^3\rho\|_{L^2_x}\right) + \left( \|\n_x^2 h_0\|_{L^4_x}+\v\|\n_x^2\rho\|_{L^4_x}\right) \left(\|\n_x h_0\|_{L^4_x}+\v\|\n_x\rho\|_{L^4_x}\right)\\
&\ + \left( \|\n_x h_0\|_{L^6_x}+\v\|\n_x\rho\|_{L^6_x} \right)^3\\
\lesssim&\ \sum_{i=1}^3 \left( \|\n_x h_0\|_{H^2_x}^i+\v^i\|\n_x\rho\|_{H^2_x}^i \right)\\
\lesssim&\ \|\n_x h_0\|_{H^2_x}+\v\|\n_x\rho\|_{H^2_x},\\[4pt]
\| \n_x^4 \left(\frac{1}{1+h+\v\rho}\right) \|_{L^2_x}
\lesssim&\ \left(\|\n_x^4 h_0\|_{L^2_x} + \v\|\n_x^4\rho\|_{L^2_x}\right) + \left(\|\n_x^3 h_0\|_{L^4_x}+\v\|\n_x^3\rho\|_{L^4_x}\right) \left(\|\n_x h_0\|_{L^4_x}+\v\|\n_x\rho\|_{L^4_x}\right)\\
& + \left(\|\n_x^2 h_0\|_{L^4_x}+\v\|\n_x^2\rho\|_{L^4_x}\right)^2 + \left(\
\|\n_x h_0\|_{L^\infty_x}+\v\|\n_x\rho\|_{L^\infty_x}\right) \left(\|\n_x h_0\|_{L^6_x}+\v\|\n_x\rho\|_{L^6_x}\right)^3 \\
\lesssim& \ \sum_{i=1}^4 \left( \|\n_x h_0\|_{H^3_x}^i+\v^i\|\n_x\rho\|_{H^3_x}^i \right)\\
\lesssim&\ \|\n_x h_0\|_{H^3_x} + \v\|\n_x\rho\|_{H^3_x},\\
\end{aligned}
\end{equation*}

For any $\tilde{\g}\in \R$, we first claim that
\begin{equation}\label{estimate of bar-gamma-h}
\left| (1+h_0)^{\tilde{\g}}-1 \right| \lesssim |h_0|,
\end{equation}
for $B_{\tilde{\g}}^0$,
\begin{equation}\label{estimate of bar-gamma-rho}
\left|B^0_{\tilde{\g}}\right| = \left|(1+h_0+\v\rho)^{\tilde{\g}}-(1+h_0)^{\tilde{\g}}-\tilde{\g}\v(1+h_0+\v\rho)^{\tilde{\g}-1}\rho \right| \lesssim  \v^2 |\rho^2|,
\end{equation}
and
\begin{equation}\label{estimate of 1+h0+rho-2}
\left| (1+h_0+\v\rho)^{\tilde{\g}} - (1+h_0)^{\tilde{\g}}\right| \lesssim  \v|\rho|.
\end{equation}

We prove the claims above by the Taylor expansion.

If $\tilde{\g}=0$ or $\tilde{\g}=1$, \eqref{estimate of bar-gamma-h} and \eqref{estimate of bar-gamma-rho} are correct.

If $\tilde{\g}\neq 0$ and $\tilde{\g}\neq 1$. To prove \eqref{estimate of bar-gamma-h}, we denote $\psi_1(s)=(1+s h_0)^{\tilde{\g}}-1$ with $0\leq s\leq 1$ and then
\begin{equation*}
\begin{aligned}
\psi'_1(s)=&\ \tilde{\g}(1+sh_0)^{\tilde{\g}-1}h_0,\\[4pt]
\psi''_1(s)=&\ \tilde{\g}(\tilde{\g}-1)(1+sh_0)^{\tilde{\g}-2}h_0^2.
\end{aligned}
\end{equation*}
By furthering using the Taylor expansion, we find that there exists a constant $\xi_1\in(0,1)$ such that
\begin{equation*}
\begin{aligned}
|(1+h_0)^{\tilde{\g}}-1| = |\psi_1(1)-\psi_1(0)| =& |\psi'_1(0)+\frac{1}{2}\psi''_1(\xi_1)|\\
\lesssim &\  |h_0|+|h_0|^2\\
\lesssim&\ \left(|h_0|+\sup_{0\leq t\leq T}\|h_0\|_{L_x^\infty}|h_0| \right) \lesssim \  |h_0|.
\end{aligned}
\end{equation*}

To prove the inequality \eqref{estimate of bar-gamma-rho}, we denote $\psi_2(s)=(1+h_0+s\v\rho)^{\tilde{\g}}$ with $0\leq s\leq 1$ and then
\begin{equation*}
\begin{aligned}
\psi'_2(s)=&\tilde{\g}(1+h_0+s\v\rho)^{\tilde{\g}-1}\v\rho,\\[4pt]
\psi''_2(s)=&\tilde{\g}(\tilde{\g}-1)(1+h_0+s\v\rho)^{\tilde{\g}-2}\v^2\rho^2.
\end{aligned}
\end{equation*}
Therefore, still using the Taylor expansion, there exists a constant $\xi_2,\xi_3\in(0,1)$ such that
\begin{equation*}
\begin{aligned}
|B^0_{\tilde{\g}}|&\ =|(1+h_0+\v\rho)^{\tilde{\g}}-(1+h_0)^{\tilde{\g}}-\tilde{\g}\v(1+h_0+\v\rho)^{\tilde{\g}-1}\rho|\\
&\ =|\psi_2(1)-\psi_2(0)-\psi_2'(1)|\\
&\ =|(\psi'_2(0)-\psi_2'(1))+\frac{1}{2}\psi''_2(\xi_2)| \\
&\ =|\psi_2''(\xi_3) + \frac{1}{2}\psi''_2(\xi_2)|\\
&\ \lesssim \v^2|(1+h_0+\xi_2\v\rho)^{\tilde{\g}-2}+(1+h_0+\xi_3\v\rho)^{\tilde{\g}-2}||\rho^2|\\
&\ \lesssim \v^2|\rho^2|,
\end{aligned}
\end{equation*}
where \eqref{A priori assumption} is used.

Observing the proof of the claim \eqref{estimate of bar-gamma-rho} and using \eqref{Key inequality}, it is easy to obtain \eqref{estimate of 1+h0+rho-2}.

By applying the estimate \eqref{estimate of rho0-rho-k} and \eqref{estimate of bar-gamma-h}, we have
\begin{equation*}
\begin{aligned}
\|(1+h_0)^{\g-1}-1\|_{L^\infty_x}&\lesssim \|h_0\|_{L^\infty_x}\lesssim \|h_0\|_{H^2_x},\\[4pt]
\|(1+h_0)^{\g-1}-1\|_{H^1_x}&\leq \|(1+h_0)^{\g-1}-1\|_{L^2_x}+\|\n_x[(1+h_0)^{\g-1}-1]\|_{L^2_x}\\
&\lesssim \|h_0\|_{L^2_x} + \|\n_x h_0\|_{L^2_x}\\
&\lesssim \|h_0\|_{H^1_x}.
\end{aligned}
\end{equation*}

By \eqref{estimate of rho0-rho-k}, \eqref{estimate of 1+h0+rho-2} and \eqref{A priori assumption}, we have
\begin{equation*}
\begin{aligned}
&\ \|(1+h_0+\v\rho)^{\g-1}-(1+h_0)^{\g-1}\|_{L^\infty_x}\lesssim \v\|\rho\|_{L^\infty_x} \lesssim \v\|\rho\|_{H^2_x},\\[4pt]
& \ \|(1+h_0+\v\rho)^{\g-1}-(1+h_0)^{\g-1}\|_{H^1_x}\\[4pt]
\leq & \  \|(1+h_0+\v\rho)^{\g-1}-(1+h_0)^{\g-1}\|_{L^2_x} + \left\|\n_x\big[(1+h_0+\v\rho)^{\g-1}-(1+h_0)^{\g-1}\right] \|_{L^2_x}\\[4pt]
\lesssim & \ \left[\v\|\rho\|_{L^2_x}+\|(1+h_0+\v\rho)^{\g-2}-(1+h_0)^{\g-2}\|_{L^2_x}\|\n_x h_0\|_{L^\infty_x}+\v\|\n_x\rho\|_{L^2_x}\right]\\[4pt]
\lesssim & \ \v\|\rho\|_{H^1_x}.
\end{aligned}
\end{equation*}

Notice that
\begin{equation}\label{nabla-x-1+h+rho}
\begin{aligned}
B_{\g}^1=&\ \g B_{\g-1}^0\n_x h_0+P_{2,11}\v\rho\n_x h_0, \\[6pt]
B_{\g}^2=&\ \g B_{\g-1}^0 \n_x^2 h_0+\g(\g-1)B_{\g-2}^0(\n_x h_0)^2 + \big[P_{2,21}\n_x^2 h_0\rho+P_{2,22}(\n_x h_0)^2\rho+P_{2,23}\n_x h_0\n_x\rho \\[4pt]
&\ +P_{2,24} (\n_x\rho)^2\v\big]\v, \\[6pt]
B_{\g}^3=&\ \g B_{\g-1}^0 \n_x^3 h_0+\g(\g-1)B_{\g-2}^0\n_x^2 h_0\n_x h_0+\g(\g-1)(\g-2)B_{\g-3}^0(\n_x h_0)^3 \\[4pt]
&\ +\big[P_{2,31}\n_x^3 h_0\rho+P_{2,32}\n_x^2 h_0 \n_x h_0\rho+P_{2,33}(\n_x h_0)^3\rho+P_{2,34}\n_x^2 h_0\n_x\rho+P_{2,35}\n_x h_0\n_x^2\rho \\[4pt]
&\ +P_{2,36}\n_x h_0(\n_x\rho)^2\v+P_{2,37}(\n_x\rho)^3\v^2\big]\v,\\[6pt]
B_\g^4=&\ \g B_{\g-1}^0 \n_x^4 h_0+\g(\g-1)B_{\g-2}^0\n_x^3 h_0\n_x h_0+\g(\g-1)(\g-2)B_{\g-3}^0
(\n_x^2 h_0)^2\\[4pt]
&\ +\g(\g-1)(\g-2)(\g-3)B_{\g-4}^0(\n_x h_0)^4+\big[P_{2,41}\n_x^4h_0\rho+P_{2,42}\n_x^3h_0\n_x h_0\rho+P_{2,43}(\n_x^2 h_0)^2\rho\\[4pt]
&\ +P_{2,44}(\n_x h_0)^4\rho+P_{2,45}\n_x^3 h_0\n_x\rho+P_{2,46}\n_x^2 h_0\n_x^2\rho+P_{2,47}\n_x h_0\n_x^3\rho+P_{2,48}\n_x^2 h_0(\n_x\rho)^2\v\\[4pt]
&\ +\v P_{2,49}(\n_xh_0)^2(\n_x\rho)^2+\v P_{2,50}\n_xh_0\n_x^2\rho\n_x\rho+\v^2P_{2,51}\n_x h_0(\n_x\rho)^3+\v^3 P_{2,52}(\n_x\rho)^4\big]\v,\\
\end{aligned}
\end{equation}
where $P_{2,ij}$ are polynomials of $1+h_0+\v\rho$, i.e., $P_{2,ij}=P_{2,ij}(1+h_0+\v\rho)$ for any $i,j$ and $B^\a_{\g}$ is defined in \eqref{B-la}.

For $B_\g^0$, we have
\begin{equation}\label{Estimate of B-gamma-0}
\begin{aligned}
\|B_{\g}^0\|_{L^2_x}&\ \lesssim \  \v^2\|\rho^2\|_{L^2_x} \lesssim \ \v^2\|\rho\|_{L^3_x}\|\rho\|_{L^6_x} \lesssim \ \v\|\rho\|_{H^1_x}\|\n_x\rho\|_{L^2_x},\\[6pt]
\|B_{\g}^0\|_{L^3_x}&\ \lesssim \ \v^2\|\rho\|_{L^6_x}^2\lesssim \ \v\|\rho\|_{H^1_x}\|\n_x\rho\|_{L^2_x},
\end{aligned}
\end{equation}
where \eqref{estimate of bar-gamma-rho} is used.

Therefore, by using the estimates in \eqref{estimate of bar-gamma-rho} and \eqref{Estimate of B-gamma-0}, for $B_\g^1$,
\begin{equation*}
\begin{aligned}
\|B_{\g}^1\|_{L^2_x}\lesssim&\ \|B_{\g-1}^0\|_{L^\infty_x}\|\n_x h_0\|_{L^2_x}+\v\|\rho\|_{L^\infty_x}\|\n_xh_0\|_{L^2_x}\\[4pt]
\lesssim&\ \v\|\rho\|_{H^2_x}\|\n_x h_0\|_{L^2_x},\\[6pt]
\|B_{\g}^1\|_{L^3_x}\lesssim&\ \|B_{\g-1}^0\|_{L^\infty_x}\|\n_x h_0\|_{L^3_x}+\v\|\rho\|_{L^\infty_x}\|\n_x h_0\|_{L^3_x}\\[4pt]
\lesssim&\ \v\|\rho\|_{H^2_x}\|\n_x h_0\|_{H^1_x},\\
\end{aligned}
\end{equation*}
for $B_\g^2$,
\begin{equation*}
\begin{aligned}
\|B_{\g}^2\|_{L^2_x}\lesssim&\ \|B_{\g-1}^0\|_{L^\infty_x}\|\n_x^2h_0\|_{L^2_x}+\|B_{\g-2}^0\|_{L^\infty_x}\|\n_x h_0\|_{L^4_x}^2+\v\|\rho\|_{L^\infty_x}\|\n_x^2h_0\|_{L^2_x}+\v\|\rho\|_{L^6_x}\|\n_xh_0\|_{L^6_x}^2\\
&\ +\v\|\rho\|_{L^4_x}\|\n_x h_0\|_{L^4_x}+\v^2\|\n_x\rho\|_{L^4_x}\|\n_x\rho\|_{L^4_x}\\[4pt]
&\ +\v\|\n_x\rho\|_{H^1_x}\|\n_x h_0\|_{H^1_x}+\v^2\|\n_x\rho\|_{H^1_x}\|\n_x\rho\|_{H^1_x}\\[4pt]
\lesssim&\ \v\|\rho\|_{H^2_x}\|(\n_xh_0,\n_x\rho)\|_{H^1_x},\\[8pt]
\end{aligned}
\end{equation*}
\begin{equation*}
\begin{aligned}
\|B_{\g}^2\|_{L^3_x}\lesssim&\ \|B_{\g-1}^0\|_{L^\infty_x}\|\n_x^2h_0\|_{L^3_x}+\|B_{\g-2}^0\|_{L^\infty_x}\|\n_x h_0\|_{L^6_x}^2+\v\|\rho\|_{L^\infty_x}\|\n_x^2 h_0\|_{L^3_x}+\v\|\rho\|_{L^\infty_x}\|\n_x h_0\|_{L^6_x}^2\\
&\ +\v\|\rho\|_{L^6_x}\|\n_x h_0\|_{L^6_x}+\v^2\|\n_x\rho\|_{L^6_x}\|\n_x\rho\|_{L^6_x}\\[4pt]
\lesssim&\ \v\|\rho\|_{H^2_x}\|(\n_x h_0,\n_x\rho)\|_{H^2_x},\\
\end{aligned}
\end{equation*}
for $B_\g^3$,
\begin{equation*}
\begin{aligned}
\|B_{\g}^3\|_{L^2_x}\lesssim&\ \|B_{\g-1}^0\|_{L^\infty_x}\|\n_x^3 h_0\|_{L^2_x}+\|B_{\g-2}^0\|_{L^\infty_x}\|\n_x^2h_0\|_{L^4_x}\|\n_x h_0\|_{L^4_x}+\|B_{\g-3}^0\|_{L^\infty_x}\|\n_x h_0\|_{L^6_x}^3\\
& \ +\v\|\rho\|_{L^\infty_x}\|\n_x^3 h_0\|_{L^2_x}+\v\|\rho\|_{L^6_x}\|\n_x^2 h_0\|_{L^6_x}\|\n_x h_0\|_{L^6_x}+\v\|\rho\|_{L^\infty_x}\|\n_x h_0\|_{L^6_x}^3+\v\|\n_x\rho\|_{L^4_x}\|\n_x^2 h_0\|_{L^4_x}\\
& \ +\v\|\n_x^2\rho\|_{L^4_x}\|\n_x\rho\|_{L^4_x} +\v^2\|\n_x\rho\|_{L^6_x}^2\|\n_x h_0\|_{L^6_x}+\v^3\|\n_x\rho\|_{L^6_x}^3\\[4pt]
\lesssim&\ \v\|\rho\|_{H^2_x}\|(\n_xh_0,\n_x\rho)\|_{H^2_x},\\[8pt]
\|B_{\g}^3\|_{L^3_x}\lesssim&\ \|B_{\g-1}^0\|_{L^\infty_x}\|\n_x^3h_0\|_{L^3_x}+\|B_{\g-2}^0\|_{L^\infty_x}\|\n_x^2 h_0\|_{L^6_x}\|\n_x h_0\|_{L^6_x}+\|B_{\g-3}^0\|_{L^\infty_x}\|\n_x h_0\|_{L^\infty}\|\n_x h_0\|_{L^6_x}^2\\
&\ +\v\|\rho\|_{L^\infty_x}\|\n_x^3 h_0\|_{L^3_x}+\v\|\rho\|_{L^\infty_x}\|\n_x^2 h_0\|_{L^6_x}\|\n_x h_0\|_{L^6_x}+\v\|\rho\|_{L^\infty_x}\|\n_x h_0\|_{L^\infty_x}\|\n_x h_0\|_{L^6_x}^2\\
&\ +\v\|\n_x\rho\|_{L^6_x}\|\n_x^2 h_0\|_{L^6_x}+\v\|\n_x^2\rho\|_{L^6_x}\|\n_x\rho\|_{L^6_x}+\v^2\|\n_x\rho\|_{L^\infty_x}\|\n_x\rho\|_{L^6_x}\|\n_x h_0\|_{L^6_x}\\
&\ +\v^3\|\n_x\rho\|_{L^\infty_x}\|\n_x\rho\|_{L^6_x}^2\\[4pt]
\lesssim&\ \v\|\rho\|_{H^3_x}\|(\n_xh_0,\n_x\rho)\|_{H^3_x},\\
\end{aligned}
\end{equation*}
and for the last term $B_{\g}^4$,
\begin{equation*}
\begin{aligned}
\|B_{\g}^4\|_{L^2_x}\lesssim&\ \|B_{\g-1}^0\|_{L^\infty_x}\|\n_x^4 h_0\|_{L^2_x}+\|B_{\g-2}^0\|_{L^\infty_x}\|\n_x^3 h_0\|_{L^4_x}\|\n_x h_0\|_{L^4_x}+\|B_{\g-3}^0\|_{L^\infty_x}\|\n_x^2 h_0\|_{L^4_x}^2\\
&\ +\|B_{\g-4}^0\|_{L^\infty_x}\|\n_x h_0\|_{L^\infty_x}\|\n_x h_0\|_{L^6_x}^3+\v\|\rho\|_{L^\infty_x}\|\n_x^4 h_0\|_{L^2_x}
+\v\|\rho\|_{L^6_x}\|\n_x^3 h_0\|_{L^6_x}\|\n_x h_0\|_{L^6_x}\\
&\ +\v\|\rho\|_{L^6_x}\|\n_x^2 h_0\|_{L^6_x}^2+\v\|\rho\|_{L^\infty_x}\|\n_x h_0\|_{L^\infty_x}\|\n_xh_0\|_{L^6_x}^3+\v\|\n_x\rho\|_{L^\infty_x}\|\n_x^3\rho\|_{L^2_x}+\v\|\n_x^2\rho\|_{L^4_x}\|\n_x^2 h_0\|_{L^4_x}\\
&\ +\v\|\n_x^3\rho\|_{L^2_x}\|\n_x h_0\|_{L^\infty_x}+\v^2\|\n_x\rho\|_{L^6_x}^2\|\n_x^2 h_0\|_{L^6_x}+\v^2\|\n_x^2\rho\|_{L^6_x}\|\n_x h_0\|_{L^6_x}\\
&\ +\v^2\|\n_x\rho\|_{L^6_x}\|\n_x^2\rho\|_{L^6_x}\|\n_x h_0\|_{L^6_x}+\v^3\|\n_xh_0\|_{L^\infty_x}\|\n_x\rho\|_{L^6_x}^3+\v^4\|\n_x\rho\|_{L^\infty_x}\|\n_x\rho\|_{L^6_x}^3\\[4pt]
\lesssim&\ \v\|\rho\|_{H^3_x}\|(\n_x h_0,\n_x\rho)\|_{H^3_x},\\
\end{aligned}
\end{equation*}
where \eqref{Key inequality} and \eqref{A priori assumption} are utilized.

\end{proof}

\section*{Acknowledgment}
H.~Wen was partially supported by the National Natural Science Foundation of China $\#12471209$ and by Guangzhou Basic and Applied Research Projects SL2024A04J01206.

\section*{Data Availability Statement}
Data sharing not applicable to this article as no datasets were generated or analysed during the current study.

\end{document}